\documentclass[12pt]{article} 

\usepackage[width=175mm,top=25mm,bottom=25mm]{geometry} 
\usepackage{amsmath}
\usepackage[utf8]{inputenc}  
\usepackage{algorithm}
\usepackage{algorithmic}
\usepackage{amssymb} 
\usepackage{bbm}
\usepackage{subcaption}
\usepackage{mathptmx} 
\usepackage{mathtools} 
\usepackage{stackengine} 
\usepackage{hyperref}
\usepackage{amsthm} 
\usepackage[noabbrev]{cleveref}
\crefname{ineq}{inequality}{inequalities}
\creflabelformat{ineq}{(#2{\upshape#1}#3)}
\crefname{probl}{problem}{problems}
\creflabelformat{probl}{(#2{\upshape#1}#3)}
\usepackage{todonotes}
\usepackage{comment}
\usepackage{tikz}

\newtheorem{prob}{Problem}[section] 
\newtheorem{lem}[prob]{Lemma} 
\newtheorem{thm}[prob]{Theorem} 
\newtheorem{defin}[prob]{Definition}

\newtheorem{prop}[prob]{Proposition}
\newtheorem{rem}[prob]{Remark}
\newtheorem{assump}[prob]{Assumption}
\newtheorem{coro}[prob]{Corollary}
\newtheorem{examp}[prob]{Example}
\newtheorem{nota}[prob]{Notation}

\crefname{prob}{problem}{problems}
\crefname{lem}{lemma}{lemmas} 
\crefname{thm}{theorem}{theorems} 
\crefname{defin}{definition}{definitions}
\crefname{memo}{memo}{memos}
\crefname{prop}{proposition}{propositions}
\crefname{rem}{remark}{remarks}
\crefname{coro}{corollary}{corollaries}
\crefname{examp}{example}{examples}
\crefname{nota}{notation}{notations}

\newcommand{\R}{\mathbb{R}}
\newcommand{\adm}{\mathcal{A}} 
\newcommand{\fid}{\mathcal{F}} 
\newcommand{\for}{K} 
\newcommand{\reg}{\mathcal{R}} 
\newcommand{\cad}{c\`{a}dl\`{a}g}

\newcommand{\DE}{\mathcal{D}_E}

\newcommand{\C}{\Omega}
\newcommand{\T}{\mathcal{T}_*}
\newcommand{\probs}{\mathcal{P}}
\newcommand{\wass}{\mathcal{W}_1(\Omega)}
\newcommand{\B}{\mathcal{B}}
\newcommand{\V}{L_w^2(I;\M(\C))}
\newcommand{\Vstar}{L_w^2(I;C(\C))}
\newcommand{\E}{\mathcal{E}}
\newcommand{\M}{\mathcal{M}}
\newcommand{\D}{\mathcal{D}}

\newcommand{\sol}{\mathcal{S}}
\newcommand{\ext}{\textup{ext}}
\newcommand{\lin}{\textup{lin}}
\newcommand{\rec}{\textup{rec}}
\newcommand{\Leb}{\mathcal{L}}

\newcommand{\ws}{\mathrel{\ensurestackMath{\stackon[1pt]{\rightharpoonup}{\scriptstyle\ast}}}} 
\newcommand{\mres}{\mathbin{\vrule height 1.3ex depth 0pt width 0.13ex\vrule height 0.13ex depth 0pt width 1ex}} 
\DeclareMathOperator*{\esssup}{ess\,sup}

\DeclareMathOperator*{\essvar}{ess\,var}
\DeclareMathOperator*{\argmax}{arg\,max}
\DeclareMathOperator*{\argmin}{arg\,min}
\DeclareMathOperator*{\var}{var}
\newcommand{\J}{\mathrm{D}}
\newcommand{\e}{\mathrm{d}}

\usepackage[backend=biber,maxnames=5,style=numeric]{biblatex}
\title{Sparsity for dynamic inverse problems on Wasserstein curves with bounded variation}
\author{Marcello Carioni\thanks{Faculty of Electrical Engineering, Mathematics and Computer Science, University of Twente, m.c.carioni@utwente.nl}\qquad\qquad\qquad Julius Lohmann\thanks{International Research Fellow of Japan Society for the Promotion of Science, Department of Mathematics, Institute of Science Tokyo, lohmann.j.aa@m.titech.ac.jp}}

\begin{document}
	\maketitle
	\noindent
	\begin{abstract}
	We investigate a dynamic inverse problem using a regularization which implements the so-called Wasserstein-$1$ distance. It naturally extends well-known static problems such as lasso or total variation regularized problems to a (temporally) dynamic setting. Further, the decision variables, realized as BV curves, are allowed to exhibit discontinuities, in contrast to the design variables in classical optimal transport based regularization techniques. We prove the existence and a characterization of a sparse solution. Further, we use an adaption of the fully-corrective generalized conditional gradient method to experimentally justify that the determination of BV curves in the Wasserstein-$1$ space is numerically implementable. 
	\end{abstract}
	\tableofcontents
	
	\section{Introduction}
    \label{Introduction}
    Consider an \textit{inverse problem} composed of a \textit{fidelity term} and \textit{regularization term}, 
        \begin{equation*}
            \inf_{v\in V}F(Av)+\alpha R(v),
        \end{equation*}
        where $\alpha\in\R_+^*$ is a \textit{regularization parameter}, $A:V\to W$ is a \textit{forward operator} from a \textit{model space} $V$ to a \textit{data space} $W$, $F:W\to\R\cup\{\infty\}$ is a \textit{fidelity}, and $R$ is a \textit{regularizer}.
        
        In this article, we are particularly interested in so-called \textit{sparse optimization} where the regularizer $R$ enforces a `sparsity property'. We briefly recall some well-known instances in this context. 
        The first is termed \textit{$\ell^1$-regularization} or \textit{lasso} \cite{T96}, for example in the formulation
        \begin{equation*}
        \inf_{v\in\R^N}\|Av-f\|_2^2+\alpha\|v\|_{1},
        \end{equation*}
        where $A\in\R^{N\times N}$, $f\in\R^N$, and $\|\cdot\|_{p}$ denotes the $\ell^p$-norm, that is $\|v\|_{p}^p=|v_1|^p+\ldots+|v_N|^p$. If $A$ is orthonormal, then it is straightforward to check that the solution is unique and given by $S_{\alpha/2}(A^\intercal f)$, where thresholding operator $S_{\alpha/2}$ is applied component-wise and given by 
        \begin{equation*}
            S_{\alpha/2}(x)=\begin{cases*}
                0&if $x\leq\alpha/2$,\\
                x-\textup{sign}(x)\alpha/2&else.
            \end{cases*}
        \end{equation*}
        Hence, the larger $\alpha$, the more zero entries in the solution, which is called \textit{sparsity} in this context. A variant of this problem is given by (see for instance \cite[Ch.\:6]{MS12}) 
        \begin{equation*}
        \inf_{v\in\R^N}\|Av-f\|_2^2+\alpha\| Mv\|_{1},
        \end{equation*}
        where $M\in\R^{M\times N}$. For example, one can take matrix $M=\bigl(\begin{smallmatrix}
            M_1\\ M_2
        \end{smallmatrix}\bigl)\in\R^{2N\times N}$, with $M_1,M_2\in\R^{N\times N}$, as a discretization of the two-dimensional gradient operator applied to some grayscale image with $N=N_1N_2$ pixels stored in the vector $v$. Then the problem can be seen as a discretization of so-called \textit{total variation regularization} (cf.\:\cite{ROF92}). The regularizer $R(v)=\| Mv\|_{1}$ then enforces sparsity of the (discrete) gradient which particularly yields sharp edges in the image (which is to be reconstructed from possibly noisy data $f$). A solution can be obtained by reformulating the problem as a quadratic program \cite{LS04}. In a spatially continuous setting, one may replace $\R^N$ by some non-empty bounded Lipschitz domain $U\subset\R^n$ and regularizer $v\mapsto\| Mv\|_{1}$ by the total variation of a BV function \cite{bredies2020sparsity},
        \begin{equation*}
            \inf_{v\in BV(U)}F(Av)+|\mathrm{D}v|(U),
        \end{equation*}
        where we assume that $A:BV(U)\to\R^m$ is continuous and linear with $A(BV(U))=\R^m$ and $F:\R^m\to\R\cup\{\infty\}$ is coercive, convex, lower semicontinuous, and proper. In this setting, there exists a sparse solution \cite[Thm.\:4.8]{bredies2020sparsity}
        \begin{equation*}
            v^{opt}=c+\sum_{i=1}^L\frac{c^i}{|\mathrm{D}1_{E_i}|(U)}1_{E^i},
        \end{equation*}
        where $c\in\R,c_i\in\R_+^*,L\leq\textup{dim}(\R^m\slash A(\R))$, and the $E_i\subset U$ are \textit{simple} (see \cite[Def.\:4.4 \& 4.5]{bredies2020sparsity}), in particular each $E^i$ cannot be decomposed into two sets with positive volume such that the sum of their perimeters equals the perimeter of $E^i$, that is $|\mathrm{D}1_{E^i}|(U)$ ($1_{E^i}$ denoting the characteristic function of $E^i$). Hence, sparsity in this configuration can be interpreted as being piecewise constant, where `piecewise' is relative to the simple sets $E^i$ (respectively their intersections). Recall that, by the Riesz representation theorem and the definition of $|\mathrm{D}v|(U)$, we can interpret the distributional gradient of $v\in BV(U)$ as a vector-valued Radon measure $\mathrm{D}v\in\mathcal{M}(U;\R^n)$. For general vector-valued Radon measures, one may consider the problem in \cite{bredies2013inverse},
        \begin{equation*}
        \inf_{v\in \mathcal{M}(X;\R^k)}\|Av-g\|_H^2+\alpha\|v\|_\mathcal{M},
        \end{equation*}
        where $X$ is a locally compact and separable metric space, $H$ is a Hilbert space ($g\in H$), $A:\mathcal{M}(X;\R^k)\to H$ is the adjoint of some continuous and linear $A_*:H\to C_0(X;\R^k)$, and $\|\cdot\|_\mathcal{M}$ denotes the total variation norm. For $X\subset\R^n$ non-empty, bounded, and open, $k=1$, finite-dimensional $H$, and $A(\mathcal{M}(X;\R))=H$, there exists a minimizer which can be written \cite[Thm.\:4.2]{bredies2020sparsity}
        \begin{equation*}
         v^{opt}=\sum_{i=1}^N\lambda^i\delta_{x^i},
        \end{equation*}
        where $N\leq\textup{dim}(H),\lambda^i\in\R$, and $x^i\in X $ ($\delta_{x^i}$ denotes the Dirac measure centered at $x^i$).
        
        In this paper, we investigate sparsity for a \textit{dynamic inverse problem} whose regularizer is related to the above problems. As a motivation, recall that the above regularizers enforce the decision variables to be concentrated on few simple geometric objects, for example time points of a discrete signal $v\in\R^N$ if $R(v)=\|v\|_{1}$, spikes of $v\in\M(X;\R)$ if $R(v)=\|v\|_{\M}$, adjoint pixels of $v\in\R^N$ if $R(v)=\|M_1v\|_{1}+\|M_2v\|_{1}$, or simple sets in the support of $v\in BV(U)$ if $R(v)=|\mathrm{D}v|(U)$. In the dynamic, time-dependent case, this \textit{simpleness} should be reflected in objects which evolve over time. Another property of the above regularizers is that they separate the contributions of the different geometric objects: $\ell^1$-norm $R=\|\cdot\|_{\ell^1}$ is a sum over the entries, total variation $R=\|\cdot\|_{\M}$ is a supremum over partitions (this also applies to $R(v)=|\mathrm{D}v|(U)$ if $\mathrm{D}v$ is interpreted as a vector-valued measure). Therefore, this \textit{separating} property should be observed in the dynamics as well.
        
        A natural setup which implements the above mentioned simpleness and separating can be obtained by considering a well-known distance from optimal transport theory: We have seen that $R(v)=\|v\|_\mathcal{M}(X)$ is used in the case of vector-valued Radon measures $v\in \mathcal{M}(X;\R^k)$. This norm actually appears in the Beckmann formulation of the \textit{Wasserstein-$1$ distance} \cite[Thm.\:4.6]{San} between (compactly supported) probability measures $\rho^+,\rho^-\in\probs(\R^n)$,
    \begin{equation*}
    W_1(\rho^+,\rho^-)=\min_v\|v\|_\M,  
    \end{equation*}
    where the minimum is over vector-valued Radon measures $v\in\M(\R^n;\R^n)$ with distributional divergence equal to $\rho^+-\rho^-$. One can restrict this minimization to compactly supported and `loop-free' $v$ and each candidate minimizer can be interpreted as a continuous curve $t\mapsto\mu_t\in(\probs(\R^n),W_1)$ (see \cref{Discussion}), where $t$ represents a time variable. Such $\mu$ and their evolution with respect to $W_1$ are a reasonable choice for a dynamic inverse problem implementing the above characteristics which we will briefly clarify. First, the Wasserstein-$1$ distance favors trajectories on which there is no direct interaction between the different mass particles which reflects the desired simpleness. More specifically, the Wasserstein-$1$ distance minimizes $\int|x-y|\e\pi(x,y)$ over $\pi$, where $|x-y|\e\pi(x,y)$ is the (infinitesimal) transportation cost (in terms of Euclidean distance) of moving (infinitesimal) amount of mass $\e\pi(x,y)$ from $x$ to $y$ --- there is no efficiency gain when particles bundle because the transportation cost is linear in the mass. Regularization with variation $\textup{var}(\mu)$ would ensure that the total accumulated travel distance of all mass particles is kept small. In particular, if the curve $\mu$ is only concentrated in one (Euclidean) curve in $\R^n$, then it will tend to be close to a traverse in the corresponding dynamic inverse problem (in a discrete setting with data measured at finitely many time points). Second, the separating property is given by Smirnov's decomposition \cite[Thm.\:C]{S}: one can write the total variation $\|v\|_\M$ (where $v$ can be interpreted as a normal $1$-current in $\R^n$ whose boundary is equal to $\rho^--\rho^+$) as an integral over the contributions of the different particle trajectories. We do not only allow for continuous curves (which may show diffusive behaviour), but also curves with jumps. The possibility to jump allows for the adjustment to heavily scattered data even on a small time scale --- however, such jumps are penalized through the regularization. This approach (and the resulting weak assumption on the regularity of a curve) extends the optimal transport based regularization techniques used so far, as we will highlight in more detail below. Finally, as in the above examples, the total mass should also appear in the regularizer. Hence, we consider $\R_+\probs(\R^n)$ instead of $\probs(\R^n)$ (the Wasserstein-$1$ distance apparently just rescales). Our proposed regularizer of $\mu\in \R_+BV([0,T];(\probs(\C),W_1))$ ($\Omega\subset\R^n$ compact and convex, $[0,T]$ some time interval) is given by
    \begin{equation*}
    T\omega(\mu)+\essvar(\mu),
    \end{equation*}
    where $T\omega(\mu)\in\R_+$ represents the time integral of the temporally constant mass $\omega(\mu)$ of $\mu$ and $\essvar(\mu)$ is its essential variation (for convenience, we normalize the time interval and add regularization parameters to both terms). We show that this regularizer enforces a sparsity property similar to those listed above. In particular, we prove the existence and representation of a sparse minimizer for the corresponding inverse problem under mild assumptions. We also provide numerical experiments based on an adaption of the algorithm proposed in \cite{bredies2024asymptotic} (employing our sparse characterization). We use a discretization procedure (via temporal deblurring) which suits to our class of admissible paths. Our numerical experiments demonstrate that sparse or diffuse ground truths with jumps can be accurately reconstructed.
    
    We mention that in \cite{BBFA23} the authors study a regularizer defined on normal $1$-currents in some non-empty, bounded, and open subset of $\R^n$. The regularizer is given as the sum of mass (which is equal to $\|\cdot\|_\M$) and boundary mass. In comparison, this regularizer acts on (static) $1$-currents and the penalty on the boundary may be seen as a quantization of the number of jumps.
    
    In general, the main goal in the modeling of dynamic inverse problems is to incorporate knowledge about the behaviour of a dynamic source which is to be reconstructed from possibly noisy measurements. A crucial challenge is to \textit{correlate} the observations given at different time instances. The choice of an appropriate regularizer is often a non-trivial task considering, inter alia, the available data, measurement procedure, or physical laws. For example, in \cite{schmitt2002efficient, schmitt2002efficient2} the authors require `temporal smoothness' to model observed displacements in X‑ray computed tomography or current density reconstruction. 
    
    In recent years, temporal regularization using optimal transport theory has gained a lot of interest. The underlying assumption is that the dynamics under consideration follow certain physical principles. Examples of related inverse problems can be found in positron emission tomography (PET) \cite{SSW20, dawood2010continuity}, magnetic resonance imaging (MRI) \cite{BF20}, single-particle tracking \cite{KCFR22, duval2024dynamical}, and particle image velocimetry (PIV) \cite{saumier2015optimal}.
        These works use the so-called Benamou\textendash Brenier formulation of the Wasserstein-$2$ distance $W_2(\rho^+,\rho^-)$ between $\rho^+,\rho^-\in\probs(\C)$ \cite{BB00},
        \begin{equation*}
            W_2(\rho^+,\rho^-)^2 = \inf_{\rho,v} \:T\int_{[0,T]} \int_{\Omega} |v_t|^2\e\rho_t\e t,
        \end{equation*}
        where the infimum is over sufficiently regular time-dependent mass distributions $\rho=\rho_t$ and velocity fields $v=v_t$ on $\Omega$ which satisfy the continuity equation $\partial_t\rho+\textup{div}(\rho v)=0$ and initial respectively final condition $\rho_0=\rho^+$ and $\rho_T=\rho^-$ (otherwise, the squared Wasserstein-$2$ distance can as well be written as an infimum over transport plans $\pi$ as above, with $|x-y|$ replaced by $|x-y|^2$). Note that $\frac{1}{2T}W_2(\rho^+,\rho^-)^2$ can be interpreted as an infimization of the time integral over the (total) kinetic energy at time $t$. Inspired by this result, a dynamic regularizer has been devised, for example in the form (with $T=1$)
        \begin{equation*}
            R(p,\mu) =  \|\mu\|_{\mathcal{M}}+\int_{(0,1)\times\Omega} \left|\frac{\e p}{\e\mu}(t,x)\right|^2\e\mu(t,x),
        \end{equation*}
        where $\mu$ is a nonnegative and $p$ (representing the physical momentum) a vector-valued Radon measure on $(0,1)\times\Omega$ with $p\ll\mu$, subject to the continuity equation (which, by definition of momentum, becomes $\partial_t\mu+\textup{div}(p)=0$), cf.\:\cite{BF20, KCFR22}. Note that one may formally replace the second term in $R$ by the time integral of the metric derivative of an absolutely continuous curve \cite[Thm.\:5.27]{San}
        \begin{equation*}
        t \mapsto \mu_t \in (\omega(\mu)\probs(\R^n), W_2).      
        \end{equation*}
        From the point of view of sparse optimization, it has also been shown that such dynamic regularization enforces sparsity: under reasonable assumptions (in particular, finite dimensional data space), there always exists a reconstruction of the form (after disintegration)
        \begin{equation*}
            t\mapsto\mu^{opt}_t = \sum_{i=1}^N \lambda^i \delta_{\gamma^i_t}
        \end{equation*}
        with $\lambda^i \in \R_+^*$ and absolutely continuous curves $\gamma^i:[0,1]\to\Omega$ \cite{bredies2021extremal}. This characterization, which can also be extended to the unbalanced case \cite{bredies2022superposition}, is a consequence of \cite{S}, see \cite[Thm.\:8.2.1]{AGS}. 
        Moreover, the sparse structure has been successfully used to design sparse optimization algorithms that directly optimize the curve $\mu$ by inserting or deleting curves of type $\gamma$ and optimizing the weights $\lambda$ \cite{KCFR22,duval2024dynamical}.
        A feature of the above framework is that the reconstructed curve $t\mapsto\mu_t$ respectively the underlying paths $t\mapsto\gamma^i_t$ are \textit{absolutely continuous}. As a consequence of this regularization bias, these models are only suitable for the tracking and reconstruction of dynamic sources that do \textit{not} show spatial discontinuities across time. Although these regularizations are useful in many applications (as pointed out above), this prior limits the reconstruction of a ground truth that exhibits a dynamic with jumps, which appear, for example, in (stochastic) processes (in particular, \cad\:processes) or object tracking in computer vision, see \:\cite{BSU20,LLT18,KGGQ19,BBM09} for some related problems. Jumps may also be produced by defects or not adequately calibrated measurement devices. Conversely, their detection can support the maintenance or recalibration. In this paper, we address this issue by building a mathematical framework that allows for the tracking of temporally discontinuous sources using the regularizer $\mu\mapsto T\omega(\mu)+\essvar(\mu)$ applied to a \textit{BV curve} $t\mapsto\mu_t\in(\omega(\mu)\probs(\C),W_1)$. Note that there also exists a dynamic formulation of the Wasserstein-$1$ distance \cite{A00}.
        
        Beside our choice of a regularizer, we mention that there are generalizations of the Wasserstein-$1$ distance whose use in the regularization of inverse problems might also be interesting to further investigate. In \cite{SW} the authors study a Wasserstein-$1$-type model in which the mass can vary over time, in \cite{MMT17,MMSR16} the authors introduce a multi-material transport problem (it also admits a temporally dynamic formulation \cite{JL23}) which, in the single-material case, can be interpreted as Wasserstein-$1$ transport, and in \cite{BB,BPSS,LSW,LSW22} the authors investigate the Wasserstein-$1$ distance with respect to the so-called (generalized) urban metric (for which the Euclidean distance is a special case).
        
    The paper is organized as follows. In \cref{InvTrackProb} we set up our dynamic inverse problem with the regularizer from above. Our main results are summarized in \cref{MainResults}. \Cref{Preliminaries} is devoted to some preliminaries for the proofs and numerical part. In \cref{Proofs} we show our main result, the existence of a sparse minimizer (\cref{extremalpoints,representersection}). We close it with a discussion on a generalization (\cref{Discussion}). In \cref{OptAlg} we describe the algorithm (\cref{DescrAlg}), which we use in our numerical experiments, and specify its discretization and implementation details (\cref{DerivationForward,ImplDet}). Finally, in \cref{NumRes} we display numerical results.
	\subsection{Inverse tracking problem for BV curves}
    \label{InvTrackProb}
    Let $\Omega\subset\R^n$ be non-empty, compact, and convex.
    \begin{defin}[Admissible BV curves, weight functional]
    \label{admBV}
    We define the set of \textbf{admissible BV curves} (not necessarily normalized) as
    \begin{equation*}
        \adm=\R_+BV([0,1];\wass),
    \end{equation*}
    where $\wass=(\probs(\C),W_1)$ denotes the Wasserstein-$1$ space of probability measures equipped with the Wasserstein-$1$ distance. If $\mu\in\adm$, then we write $\omega(\mu)$ for the \textbf{weight} of $\mu$, i.e.\:we have $\mu=\omega(\mu)\rho$ for some $\rho\in BV([0,1];\wass) $. 
    \end{defin}
    An example of a shortest \cad\:representative of $\rho\in BV([0,1];\wass)$ matching some given data at different time instances is given in \cref{FigExCad}.
    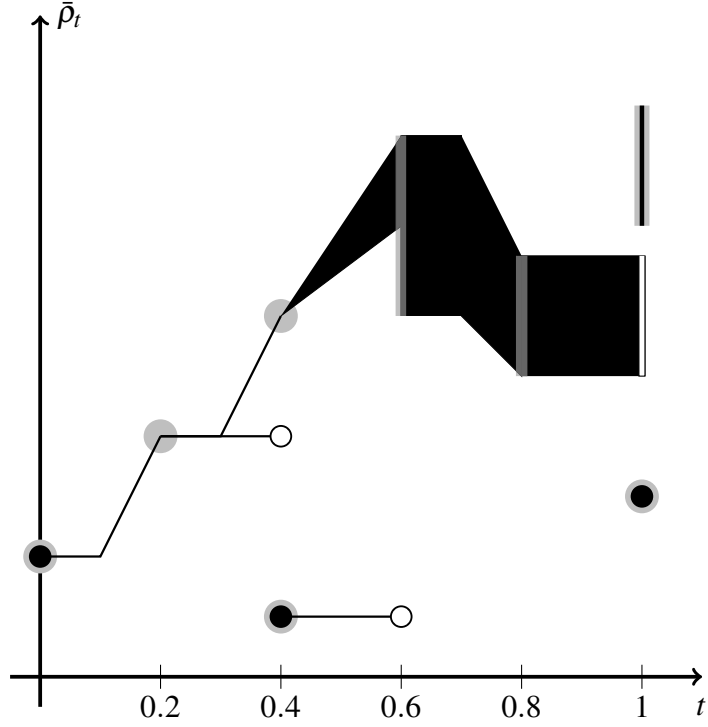
\begin{figure}[t]
    \centering
        \begin{tikzpicture}[scale=8]
        \draw[->,line width=0.5mm] (-0.05,0) -- (1.1,0);
        \draw[->,line width=0.5mm] (0,-0.05) -- (0,1.1);
        \draw (1,0.02) -- (1,-0.02);
        \node at (1,-0.05) {$1$};
        \node at (1.1,-0.05) {$t$};
        \node at (0.05,1.1) {$\bar\rho_t$};
        \node[lightgray,fill,circle,minimum size=4.5mm,inner sep=0pt] at (0,0.2) {};
        
        \draw[lightgray,line width=1.5mm] (0.8,0.5) -- (0.8,0.7);
        
        \draw (0.2,0.02) -- (0.2,-0.02);
        \node[lightgray,fill,circle,minimum size=4.5mm,inner sep=0pt] at (0.2,0.4) {};
        
        \draw (0.4,0.02) -- (0.4,-0.02);
        \node[lightgray,fill,circle,minimum size=4.5mm,inner sep=0pt] at (0.4,0.6) {};
        \node[lightgray,fill,circle,minimum size=4.5mm,inner sep=0pt] at (0.4,0.1) {};
        \draw[lightgray,line width=1.5mm] (0.6,0.6) -- (0.6,0.9);
        \filldraw[draw=black, fill=black] (0.4,0.6) -- (0.6,0.75) -- (0.6,0.9);

        \draw (0.6,0.02) -- (0.6,-0.02);
        
        \filldraw[draw=black, fill=black] (0.6,0.6) -- (0.6,0.9) -- (0.7,0.9) -- (0.7,0.6);
        \filldraw[draw=black, fill=black] (0.7,0.9) -- (0.7,0.6) -- (0.8,0.5) -- (0.8,0.7);

        \draw (0.8,0.02) -- (0.8,-0.02);
        
        \filldraw[draw=black, fill=black] (0.8,0.5) -- (0.8,0.7) -- (1.005,0.7) -- (1.005,0.5);
        \filldraw[draw=black, fill=white] (1.005,0.7) -- (1.005,0.5) -- (0.995,0.5) -- (0.995,0.7);
        
        \node[lightgray,fill,circle,minimum size=4.5mm,inner sep=0pt] at (1,0.3) {};
        
        \draw[line width=0.3mm] (0,0.2) -- (0.1,0.2) -- (0.2,0.4);

        \draw[line width=0.3mm] (0.2,0.4) -- (0.3,0.4) -- (0.4,0.6);
        \draw[line width=0.3mm] (0.2,0.4) -- (0.4,0.4);
        \draw[line width=0.3mm] (0.4,0.1) -- (0.6,0.1);

        \draw[lightgray,line width=2mm] (1,0.75) -- (1,0.95);

        \node[fill,circle,minimum size=3mm,inner sep=0pt] at (0.4,0.4) {};
         \node[white,fill,circle,minimum size=2.5mm,inner sep=0pt] at (0.4,0.4) {};

        \node[fill,circle,minimum size=3mm,inner sep=0pt] at (0.6,0.1) {};
        \node[white,fill,circle,minimum size=2.5mm,inner sep=0pt] at (0.6,0.1) {};

        \node at (0.2,-0.05) {$0.2$};
        \node at (0.4,-0.05) {$0.4$};
        \node at (0.6,-0.05) {$0.6$};
        \node at (0.8,-0.05) {$0.8$};

        \draw[gray!120,line width=0.75mm] (0.6035,0.6) -- (0.6035,0.9);
        \filldraw[draw=gray!120, fill=gray!120] (0.5915,0.887) -- (0.5915,0.7435) -- (0.6,0.75) -- (0.6,0.9);
        
        \draw[gray!120,line width=0.75mm] (0.805,0.5) -- (0.805,0.7);
        \filldraw[draw=gray!120, fill=gray!120] (0.7915,0.7) -- (0.7915,0.508) -- (0.8,0.5) -- (0.8,0.7);

        \node[fill,circle,minimum size=3mm,inner sep=0pt] at (0.4,0.1) {};
        \draw[line width=0.6mm] (1,0.75) -- (1,0.95);
        \node[fill,circle,minimum size=3mm,inner sep=0pt] at (0,0.2) {};
         \node[fill,circle,minimum size=3mm,inner sep=0pt] at (1,0.3) {};
        \end{tikzpicture}  
    \caption{Example of a shortest \cad\:curve $\bar\rho:[0,1]\to\wass$ matching given data (displayed in gray) at time points $t=0,0.2,\ldots,1$, where $\Omega\subset\R$. In the one-dimensional case, Wasserstein-$1$ transport is well-understood (by the linearity of the transportation cost combined with the fact that each particle can only move in two directions), see \cite[Ch.\:2]{San}. For example, we have $\bar\rho_t=\delta_{x_t}$ for $t\in[0,0.3]$, $\bar\rho_t=\frac{1}{2}\delta_{x_t}+\frac{1}{2}\delta_{y}$ for $t\in[0.3,0.4)$, and $\bar\rho_t=(2\Leb(\ell_t))^{-1}\Leb\mres\ell_t+\frac{1}{2}\delta_{z}$ for $t\in(0.4,0.6)$. Here $\Leb\mres\ell_t$ denotes the restriction of one-dimensional Lebesgue measure $\Leb=\Leb^1$ to line segment $\ell_t$ (which would be replaced by the one-dimensional Hausdorff measure if $n>1$). Note that the jumps at $t=0.4$ and $t=0.6$ are \cad\:(in particular, the right-hand limits correspond to the given data). At $t=1$ the curve $\bar\rho$ jumps to data consisting of a line segment and a Dirac mass. In this example, each particle (with infinitesimal small mass) moves linearly with respect to $t$ (which does not need to be satisfied), waits, or jumps.}
    \label{FigExCad}
    \end{figure}
    \begin{rem}[Structure of $\adm$]
    \label{structAdm}
    The set of admissible BV curves $\adm$ is a convex cone in the vector space $L_w^2([0,1];\M(\C))$ (see \cref{BasNot}).
    \end{rem}
    \begin{defin}[Regularizer for BV curves]
    \label{admReg}
    We define ($\alpha,\beta\in\R_+^*$ fixed)
    \begin{equation*}
    \reg_{\alpha,\beta}(\mu)=\alpha \omega(\mu)+ \beta \essvar(\mu), 
    \end{equation*}
    where $\essvar(\mu)$ denotes the essential variation of $\mu\in\adm$ (see \cref{BVetc}).
    \end{defin}
	\begin{prob}[BV curve tracking problem]
		\label{BVtp}
	    A \textbf{BV curve tracking problem} is any inverse problem of the form
		\begin{equation*}
		\inf_{\mu\in\adm} \fid(\for(\mu) )+\reg_{\alpha,\beta}(\mu),
		\end{equation*} 
        where \textbf{forward operator} $K$ maps each admissible BV curve to an element of some \textbf{data space} $Y$ and $\mathcal{F}$ is a \textbf{fidelity}. In particular, we make the following assumptions (which ensure existence, cf.\:\cref{existence}):
        \begin{itemize}
            \item $Y$ is a Hilbert space,
            \item $\mathcal{F} : Y \rightarrow \R\cup\{ \infty \}$ is bounded from below, proper, and weakly lower semi-continuous,
            \item $\for : \mathcal{A} \rightarrow Y$ is continuous in the following sense: if $(\mu^j)\subset\adm$ and $\mu\in\adm$ with $\mu^j_t\ws\mu_t$ for a.e.\:$t\in[0,1]$ (in short, we have $\mu^j\ws\mu$ a.e.), then $\for(\mu^j)\rightharpoonup\for(\mu)$ (weakly) in $Y$.
        \end{itemize}
	\end{prob}
    The existence of a minimizer follows readily from the direct method in the calculus of variations and \cite[Thm.\:2.4 (i)]{A90} (see beginning of \cref{Proofs}).
    \begin{prop}[Existence]
		\label{existence}
		\Cref{BVtp} admits a solution.
	\end{prop}
    \begin{rem}[Stability]
    \label{stability}
   A natural choice (which for $Y=\R^m$ corresponds to least square fitting) for $\fid$ is $\fid_f(y)=\frac{1}{2}\|y-f\|_Y^2$, where $f\in Y$ is some given reference data. In this setting, which we will use in the numerical part (\cref{OptAlg}), the following stability property is satisfied: Let $(f^j)\subset Y$ with $f^j\to \hat f$ (strongly) in $Y$. Further, assume that each $\mu^j$ is a solution of \cref{BVtp} with $\fid=\fid_{f^j}$. Then, up to a subsequence, we have $\mu^j\ws\hat\mu$ a.e.\:for some $\hat\mu\in\adm$ and $\hat\mu$ is a solution of \cref{BVtp} with $\fid=\fid_{\hat f}$. 
    \end{rem}
    \begin{examp}[$K$ with values in infinite-dimensional Hilbert space]
    \label{KInfVal}
    Let $\Phi\in C([0,1]\times\Omega^2)$. For any $\mu\in L^2_w([0,1],\M(\C))$ define $\widehat K\mu:\Omega\to\R$ by
    \begin{equation*}
        (\widehat K\mu)(x)=\int_{[0,1]}\int_\Omega\Phi_t(x,y)\e\mu_t(y)\e\Leb(t).
    \end{equation*}
    Let $(x_i)\subset\Omega$ with $x_i\to x\in\Omega$. For each $i$ we have $\int_\Omega\Phi_t(x_i,y)\e\mu_t(y)\leq\max\Phi\|\mu_t\|_\M$ for a.e.\:$t\in[0,1]$. Further, it holds $t\mapsto\|\mu_t\|_\M\in L^1([0,1])$ by definition of $L^2_w([0,1],\M(\C))$. Thus, twofold application of Lebesgue's dominated convergence theorem implies $(\widehat K\mu)(x_i)\to(\widehat K\mu)(x)$, hence $\widehat K\mu\in C(\Omega)$. In particular, we can interpret $\widehat K$ as a linear map $L^2_w([0,1],\M(\C))\to L^2(\Omega)$. Now define $K=\widehat K|_\adm$. Take a sequence $(\mu^j)\subset\adm$ such that $\mu^j\ws\mu$ a.e. Write $\mu^j=\omega(\mu^j)\rho^j$ and $\mu=\omega(\mu)\rho$ as in \cref{admBV}. Choosing the test function $\phi\equiv 1$ we obtain $\omega(\mu^j)\to\omega(\mu)$, in particular $\omega(\mu^j)$ is bounded. Using this, $\mu^j\ws\mu$ a.e., and dominated convergence again, we get 
    \begin{equation*}
    \int_\Omega|K\mu^j-K\mu|^2\e\Leb^n=\int_\Omega\left|\int_{[0,1]}\int_\Omega\Phi\e(\mu^j-\mu)\e\Leb\right|^2\e\Leb^n\to 0,
    \end{equation*}
    that is $K\mu^j\to K\mu$ (strongly) in $L^2(\C)$ so that $K$ satisfies the assumption in \cref{BVtp}. Thanks to the regularity of $\Phi$, we can simply apply the Arzel\`{a}\textendash Ascoli theorem to show that $\widehat K$ is a compact operator implying that the inverse problem $\widehat K\mu=f$ with data $f\in C(\Omega)$ is ill-posed in the sense of Hadamard \cite[Thm.\:3.3.2]{MS12}. Therefore, \cref{BVtp} can be seen as a sparse regularization of it. One may interpret $\widehat K\mu=f$ as a variant of the Fredholm integral equation of first kind (a classical ill-posed linear inverse problem) --- the regularity of $\Phi$ is needed to ensure that $\widehat K$ is well-defined (typically, one requires lower $L^2$-regularity of the kernel). Indeed, the (constant) curve $\mu\equiv g\Leb^n$ with $g\in L^2(\Omega)$ is an element of $L^2_w([0,1],\M(\C))$. Taking $\Phi$ independent of $t$ we have $(\widehat K\mu)(x)=\int_\Omega\Phi(x,y)g(y)\Leb^n(y)$.
    \end{examp}
    Note that forward operator $K$ cannot implement pointwise evaluation of $\mu\in L^2_w([0,1],\M(\C))$. This aspect is also addressed in the discretization in \cref{KFinVal}.
	\subsection{Main results}
    \label{MainResults}
    First, we use the superposition principle in \cite{ALS24} to characterize the extremal points of sublevel sets of the regularizer $\reg_{\alpha,\beta}$. Define 
    \begin{equation*}
    L_c^-(\reg_{\alpha,\beta})=\left\{\mu\in\adm\:\middle|\:\reg_{\alpha,\beta}(\mu)\leq c\right\} 
    \end{equation*}
    for any $c\in\R_+$.
	\begin{prop}[Extremal points of $L_c^-(\reg_{\alpha,\beta})$]
	\label{extremalpointschar}
	For every $c\in\R_+$ we have
	\begin{equation*}
	\ext (L_c^-(\reg_{\alpha,\beta}))=\{ \mu\in\adm\:|\:\reg_{\alpha,\beta}(\mu)\in\{ 0,c \},\:\bar{\mu}_t=\omega(\mu)\delta_{x_t}\textup{ for all }t\in[0,1) \},
	\end{equation*}
	where $\bar{\mu}$ is any \cad\:representative of $\mu$.
	\end{prop}
    We will then use \cref{extremalpointschar} and the representer theorem in \cite{BCCDGW19} to prove the existence of a sparse minimizer and its characterization. To this end, we require the following.
	\begin{assump}[Convex setup and finiteness of data]
	\label{linearf}
	The fidelity $\fid$ is convex. Further, the forward operator $\for$ is a linear map from $\adm$ to a finite-dimensional data space $Y=\R^m$ which can be linearly extended to $L_w^2([0,1];\M(\C))$.
	\end{assump}
	\begin{thm}[Existence of sparse solution]
		\label{representer}
		Let \cref{linearf} be satisfied. Then there exists a solution $\mu^{opt} \in \mathcal{A}$ of \cref{BVtp} with
		\begin{align*}
		\mu_t^{opt} = \sum_{i=1}^N \lambda^i \delta_{x^i_t}
		\end{align*}
		for a.e.\:$t \in [0,1]$, where
		$N\leq m=\textup{dim}(Y)$, $\lambda^i \in\R_+$, and $t \mapsto x^i_t \in BV([0,1]; \C)$ for $i=1,\ldots, N$.
	\end{thm}
    As a variant, one may fix $N$ and minimize over weights $\lambda^1,\ldots,\lambda^N \in\R_+^*$ and curves $x^1,\ldots,x^N \in BV([0,1]; \C)$. This optimization may be seen as regression by BV curves with small length and importance $\lambda^i$. The data may be represented by a (static) point cloud in $\Omega$ or recorded over time, for example, it may be given by finitely many maps $X^1,\ldots,X^{N_S}:[0,1]\to \Omega$, where each $X^i$ represents a sample path of a c\`{a}dl\`{a}g process $\{ X_t\:|\:t\in [0,1] \}$, that is almost surely every path $X :[0,1]\to\Omega$ is c\`{a}dl\`{a}g \cite[p.\:4]{P05}. Note that the sparse solution in \cref{representer} does not contain any diffusive behaviour which, in principle, is admissible in Wasserstein-$1$ transport, for example when a line segment is irrigated by a point source (see \cref{FigExCad}). 
    \begin{examp}[$K$ with values in finite-dimensional Hilbert space]
    \label{KFinVal}
    Recall the definition of forward operator $\widehat K:L^2_w([0,1];\M(\C))\to L^2(\Omega)$ from \cref{KInfVal}. Let $x_1,\ldots,x_L\in\Omega$. Instead of considering $\Phi\in C([0,1]\times\Omega^2)$ one may locate $\Phi^i\in C([0,1]\times\Omega)$ around $x_i$ for all $i=1,\ldots,L$. For example, each $\Phi^i$ may represent the sensitivity (which can be different for different time points $t$) of a sensor placed at $x_i$ convolved with some point spread function, that is $\Phi^i_t=\theta_t^i\ast\varphi_t^i$, where $\theta_t^i$ is the sensitivity (not necessarily continuous) and $\varphi_t^i$ models a physical spreading process (for example, spreading of light with $\varphi$ independent of $t$ and $i$ or spreading through diffusive material properties with $\varphi^i$ depending on $i$), see \cite[Section 3.1]{V23}. Then $\int_{[0,1]}\int_\Omega\Phi_t^i(y)\e\mu_t(y)\e\Leb(t)$ represents the (total) information collected by the sensor placed at $x_i$. In applications, this information is obtained by temporal sampling and one is typically interested in the information collected at specific time instances instead. We may include a temporal blurring which naturally circumvents the fact that $\widehat K\mu$ cannot realize pointwise evaluation of $\mu$: Let $t_0<t_1<\ldots<t_M$ be given time points in $[0,1]$. For each time $t_j$ we may consider 
    \begin{equation*}
        \int_{[0,1]}\phi_j(t)\int_\Omega\Phi_t^i(y)\e\mu_t(y)\e\Leb(t),
    \end{equation*}
    where $\phi_j\in C([0,1])$ models the temporal blur at $t_j$. In summary, we have $\Psi^i_j=\phi_j\Phi^i\in C([0,1]\times\C)$. Now define linear $\widetilde K:L^2_w([0,1];\M(\C))\to \R^{L\times (M+1)}$ by 
    \begin{equation*}
        (\widetilde K\mu)_j^i=\int_{[0,1]}\int_\Omega \Psi_j^i\e\mu\e\Leb
    \end{equation*}
    for each $i=1,\ldots,L$ and $j=0,1,\ldots,M$. As in \cref{KInfVal}, it is not difficult to see that $\widetilde K$ satisfies the continuity property in \cref{BVtp} so that we can define $K$ as the restriction of $\widetilde K$ to $\adm$. The non-locality in time may be regulated by $\phi_j$ given as the density of a truncated normal distribution (for a prioritization of time points close to $t_j$ regulated by the variance) or an indicator function $1_{[0,1]\cap[t_j-\delta_j,t_j+\delta_j]}$ (for a uniform weighting of information close to $t_j$, which still yields a well-defined operator $K$). The physical spreading described by $\varphi_t^i$ may as well be given by a Gaussian filter (which is typical for the spreading of light). 
    \end{examp}
	\section{Preliminaries}
    \label{Preliminaries}
    In this section, we introduce our notation and recapitulate the notions of BV curves and their \cad\: representatives.
	\subsection{Basic notation} 
    \label{BasNot}
    We use the following standard notation. 
    \begin{itemize} 
        \item We write $\R_+=[0,\infty)$ and $\R_+^*=(0,\infty)$. Further, we set $I=[0,1]$ for the \textit{unit interval}. We predominantly use this notation if $t\in I$ represents the time variable.
        \item $\Omega\subset\R^n$ denotes some non-empty, compact, and convex (spatial) \textit{domain}.
        \item If $F:X\to\R\cup\{ \infty \}$ and $c\in\R$, then we write $L_c^-(F)$ for the \textit{sublevel set} $\{ x\in X\:|\:F(x)\leq c \}$.
        \item We write $e_t(\gamma)\in\Omega$ for the \textit{evaluation} of a curve $\gamma:I\to\Omega$ at $t\in I$.
        \item If $X$ is a topological space, then we write $\B(X)$ for its \textit{Borel $\sigma$-algebra}. The \textit{restriction} $\mu\mres B$ of a map $\mu:\B(X)\to\R$ (e.g.\:Radon measure) to $B\in\B(X)$ is defined by 
        \begin{equation*}
        (\mu\mres B)(\tilde B) =\mu(B\cap\tilde B)
        \end{equation*}
        for all $\tilde B\in\B(X)$.
        \item The \textit{pushforward} $f_\#\mu$ of a measure $\mu$ on $X_1$ under a measurable map $f:X_1\to X_2$ is the measure defined by $f_\#\mu (M)=\mu(f^{-1}(M))$ for all measurable subsets $M\subset X_2$.
        \item We denote the one-dimensional \textit{Lebesgue measure} by $\Leb$.
        \item Let $A\subset\R$ be Lebesgue measurable. For $1\leq p\leq\infty$ we define $L^p(A)$ as the \textit{Lebesgue space} of equivalence classes of Lebesgue measurable functions $f:A\to\R$ with $\int_A|f|^p\e\Leb<\infty$ if $p<\infty$, $\esssup_A|f|<\infty$ if $p=\infty$. Further, if $(X,d)$ is a metric space, then we let $L^p(A;(X,d))$ denote the set of equivalence classes of Lebesgue measurable maps $f:A\to(X,d)$ for which $d(f,x)\in L^p(A)$ for some (and thus all if $A$ is bounded) $x\in X$. In each of the two definitions (of $L^p(A)$ and $L^p(A;(X,d))$) two maps belong to the same class if they coincide a.e.\:in $A$.
        \item We write $\probs(\Omega)$ for the space of \textit{probability measures} on $\Omega$, 
        \begin{equation*}
          \probs(\Omega)=\{ \rho:\B(\Omega)\to[0,1]\:|\:\rho\textup{ $\sigma$-additive with }\rho(\C)=1 \}.  
        \end{equation*}    
        \item The \textit{Wasserstein-$1$ space} is defined by $\wass=(\probs(\Omega),W_1)$. It is equipped with the \textit{Wasserstein-$1$ distance} $W_1$ given by
        \begin{equation*}
        W_1(\rho^+,\rho^-)=\inf_\pi\int_{\Omega\times\Omega}|x-y|\e\pi(x,y),
        \end{equation*}
        where the infimum is over \textit{transport plans} between $\rho^+$ and $\rho^-$, that is $\pi\in\probs(\Omega\times\Omega)$ with $\pi(B\times\C)=\rho^+(B)$ and $\pi(\C\times B)=\rho^-(B)$ for all $B\in\B(\C)$. The space $\wass$ is complete and separable since $\Omega$ is complete and separable \cite[Thm.\:6.18]{V09}.
        \item The space of \textit{Radon measures} on $\Omega$ (which by the Riesz representation theorem can be identified with the topological dual of the continuous functions $C(\Omega)$ with uniform norm $\|\cdot\|_\infty$) is given by 
        \begin{equation*}
            \M(\C)=\{ \mu:\B(\Omega)\to\R\:|\:\mu\textup{ $\sigma$-additive}  \}.
        \end{equation*}
        The corresponding operator norm is called \textit{total variation} and we denote it by
        \begin{equation*}
            \|\mu\|_{\mathcal{M}}=\sup_{\|\phi\|_\infty\leq 1}\int_\Omega\phi\e\mu.
        \end{equation*}
        We indicate the weak-$*$ convergence in $\M(\C)$ by $\ws$, that is $\mu^i\ws\mu$ in $\M(\C)$ if and only if $\langle\phi,\mu^i\rangle=\int\phi\e\mu^i\to\int\phi\e\mu=\langle\phi,\mu\rangle$ for all $\phi\in C(\C)$. 
        \item Let $A\subset\R$ be closed. Along with the above duality $\M(\C)=C(\C)^*$, we define $L_w^2(A;C(\C))$ as the space of equivalence classes of Lebesgue measurable maps $f:A\to C(\C)$ with $\int_A\|f\|_\infty^2\e\Leb<\infty$. The space $L_w^2(A;C(\C))$ with norm 
        \begin{equation*}
        \|\phi\|_{L_w^2}=\sqrt{\int_A\|\phi\|_\infty^2\e\Leb}
        \end{equation*}
        is a separable Banach space (see \cite[Section 8.18.1]{Edw65} or \cite[Thm.\:I.5.18]{Warga72}). Its (topological) dual is given by $L_w^{2}(A;\M(\C))$ \cite[Thm.\:8.20.3]{Edw65}, the space of equivalence classes of weakly measurable maps $f:A\to\M(\C)$\footnote{By \cite[Sections 1.11.1 \& 1.11.3]{Edw65} a map $f:I\to\M(\C)$ is weakly measurable if it is Lebesgue measurable with respect to Borel $\sigma$-algebra on $\M(\C)$ induced by the subspace topology of $\M(\C)=C(\C)^*\subset C(\C)'$. The algebraic dual $C(\C)'$ is equipped with the weak topology, i.e.\:the weakest topology for which each functional $f\mapsto\ell(f)$ is continuous ($\ell\in C(\C)'$).} with $\int_A\|f\|_\M^2\e\Leb<\infty$. The dual norm becomes
        \begin{equation*}
        \|\mu\|_{L_w^{2}}=\sqrt{\int_A\|\mu\|_\M^2\e\Leb}
        \end{equation*}
        and the dual pairing is given by $\langle\phi,\mu\rangle_{L_w^{2}}=\int_A\langle\phi,\mu\rangle\e\Leb$. In each of the two definitions (of $L_w^2(A;C(\C))$ and $L_w^{2}(A;\M(\C))$) two maps belong to the same class if they coincide a.e.\:in $A$. (The subscript in $L_w^2(A;C(\C))$ is justified by the fact that (Lebesgue) measurability and weak measurability coincide for $C(\C)$-valued maps, which follows from the Pettis measurability theorem and the separability of $C(\C)$.)
        \item Finally, we recall some terminology from convex analysis (see \cite[p.\:1]{K57} and \cite[Section 8]{R70}). Let $V$ be a real vector space and $C\subset V$ convex. An \textit{extremal point} (or extreme point) of $C$ is a point $v\in C$ such that $C\setminus\{ v \}$ is convex (this notion goes back to Minkowski). We write $\textup{ext}(C)$ for the set of extremal points of $C$. We say that $C$ is \textit{linearly closed} if the intersection of $C$ with every one-dimensional affine space of $V$ is closed. A \textit{ray} of $V$ is any set of the form $\{v + tw \:|\: t >0\}$ for $v,w \in V$ and $w\neq 0$. The \textit{lineality space} of $C$ is defined as $\textup{lin}(C):=\textup{rec}(C)\cap (-\textup{rec}(C))$, where $\textup{rec}(C)=\{ v\in V \:|\: C+\R_+^*v\subset C \}$ denotes the \textit{recession cone} of $C$.
    \end{itemize}
	\subsection{BV curves and \cad\:representatives}
    \label{BVetc}
    In this section, we recall the notions of BV curves and their \cad\:representatives in a metric space $(X,d)$ (following \cite{ALS24}). For a more comprehensive theory of BV curves in metric spaces the reader may consult \cite{A90}. In this article, we are particularly interested in $(X,d)=\wass$.
    \begin{defin}[Variation]
    The \textbf{variation} of a map $\mu : I \to (X,d)$ on $J \subset I$ is defined by
    \begin{align*}
        \var(\mu,J) = \sup \left\{\sum_{i=1}^N d(\mu_{t_i}, \mu_{t_{i-1}})\:\middle|\: t_0 < t_1<\ldots < t_N,\, t_i \in J  \right\}. 
    \end{align*}
    We abbreviate $\var(\mu)=\var(\mu,I)$.
    \end{defin}
    \begin{defin}[Essential variation, BV curves]
    The \textbf{essential variation} of $\mu\in L^1(I;(X,d))$ on $J \subset I$ is given by
    \begin{equation*}
     \essvar(\mu,J)= \inf\left\{ \var(\tilde\mu,J)\:\middle|\:\mu=\tilde\mu\textup{ a.e.}  \right\}.
    \end{equation*}
    We write $\essvar(\mu)=\essvar(\mu,I)$. If $\essvar(\mu) <\infty$, then we call $\mu$ a \textbf{BV curve} in $(X,d)$. We write $BV(I;(X,d))$ for the space of BV curves in $(X,d)$.
    \end{defin}
    \begin{rem}[BV Wasserstein curves]
    \label{WassCurves}
    We have $\R_+BV(I;\mathcal{W}_p(\Omega))\subset \R_+BV(I;\mathcal{W}_1(\Omega))=\adm$ for all $p\in[1,\infty)$, where $\mathcal{W}_p(\Omega)=(\probs(\C),W_p)$ denotes the Wasserstein-$p$ space with $W_p(\rho^+,\rho^-)^p=\inf_\pi\int|x-y|^p\e\pi(x,y)$ (infimum over transport plans $\pi$ between $\rho^+$ and $\rho^-$). This follows easily from the estimate (invoking H\"older's inequality) 
    \begin{equation*}
        W_1(\rho^+,\rho^-)\leq W_p(\rho^+,\rho^-)\leq\textup{diam}(\C)^{\frac{p-1}{p}}W_1(\rho^+,\rho^-)^{\frac{1}{p}}.
    \end{equation*}
    Hence, our set $\adm$ of decision variables contains all the (equivalence classes of) classical Wasserstein curves with arbitrary weight and bounded (essential) variation.
    \end{rem}
    The following lemma is a version of a standard result. We briefly recall the argument since we did not find it for our setting.
    \begin{lem}[$L^p$ regularity of BV curves]
     We have $BV(I;(X,d)) \subset L^p(I;(X,d))$ for every $1 \leq p \leq \infty$.
    \end{lem}
    \begin{proof}
    Let $\mu\in BV(I;(X,d))$ and $x\in X$ with $d(\mu,x)\in L^1(I)$. Pick any representative $\tilde\mu$ of $\mu$ with $\var(\tilde\mu)<\infty$. Then $d(\mu,x)\leq d(\tilde\mu,\tilde\mu_0)+d(\tilde\mu_0,x)\leq \var(\tilde\mu)+d(\tilde\mu_0,x)$ a.e. Thus, for $p<\infty$ we have (using the convexity of $t\to t^p$ and the boundedness of $I=[0,1]$) $\int_Id(\mu,x)^p\e\Leb\leq 2^{p-1}(\var(\tilde\mu)^p+d(\tilde\mu_0,x)^p)<\infty$, and for $p=\infty$ we get $\esssup_I d(\mu,x)\leq \sup_I d(\tilde\mu,x)\leq \var(\tilde\mu)+d(\tilde\mu_0,x)<\infty$.
    \end{proof}
    As in \cite{ALS24}, we use \cite[Ch.\:6]{SS05} to generalize the definition of variation measure to the metric space setting.
    \begin{defin}[Variation measure]
    Let $\mu\in BV(I;(X,d))$. Define a non-decreasing function $V:I\to\R$ by $V(t)=\essvar(\mu,(0,t))$ for all $t\in I$. The \textbf{variation measure} $|D\mu|:\B(I)\to\R_+$ of $\mu$ is defined as the Lebesgue\textendash Stieltjes measure \cite[Section 6.3.3, Thm.\:3.5]{SS05} which satisfies
    \begin{equation*}
    |D\mu|((a,b])=V(b) - V(a)
    \end{equation*}
    for all $(a,b]\subset I$ with $a<b$.
    \end{defin}
    We recall the definition of \cad\:curves. Every $\mu\in BV(I;(X,d))$ can be represented by such a curve $\bar\mu:I\to(X,d)$. Further, a \cad\:representative $\bar\mu$ is uniquely determined on $[0,1)$.
    \begin{defin}[C\`{a}dl\`{a}g curves in $(X,d)$, spaces $\mathcal{D}_E(\C),\D_W(\C)$]
    A map $\mu :I\to (X,d)$ is called a \textbf{\cad\:curve} if $\mu$ is right-continuous and the left limit $\lim_{\tilde t\nearrow t}\mu_{\tilde t}$ exists for every $t\in I$ (continue \`{a} droite, limite \`{a} gauche). We write $\mathcal{D}_E(\Omega)$ for the set of \cad\:curves $\gamma:I\to\Omega$ (with $d$ the Euclidean distance) and $\mathcal{D}_W(\Omega)$ for the \cad\:curves $\mu:I\to\wass$ (with $d=W_1$ the Wasserstein-$1$ distance).  
    \end{defin}
    \begin{rem}[Skorohod metric]
    If $(X,d)$ is complete and separable, then the space of \cad\:curves $\D=\{ \gamma:I\to(X,d)\:|\:\gamma\textup{ \cad} \}$ equipped with the \textbf{Skorohod $J_1$ metric} $d_\D$ is complete and separable \cite[Thm.\:12.2 \& p.\:131]{B} (cf.\:text passage below \cite[Thm.\:2.10]{ALS24}). Thus, this property is statisfied for both $\mathcal{D}_E(\Omega)$ and $\mathcal{D}_W(\Omega)$. We recall the definition of $d_\D$ (see construction in \cite[p.\:123-126]{B}). Let $\Lambda$ be the set of strictly increasing homeomorphisms $\lambda:I\to I$. 
    A complete metric is obtained by requiring that $\lambda$ is close to the identity in the sense that
    \begin{equation*}
    f(\lambda)=\sup_{\substack{s,t\in I\\ s<t}}\left|  \log\frac{\lambda(t)-\lambda(s)}{t-s} \right|
    \end{equation*}
    is small. The Skorohod $J_1$ metric is then given by
    \begin{equation*}
        d_\D(\gamma_1,\gamma_2)=\inf_{\lambda\in\Lambda}\max\{ f(\lambda), \|\gamma_1-\gamma_2\circ\lambda\|_\infty \}.
    \end{equation*}
    \end{rem}
 Next, we apply the compactness statement \cite[Thm.\:2.4 (i)]{A90} to our setting. Recall the definitions of admissible set $\adm$ and regularizer $\reg_{\alpha,\beta}$ (\cref{admBV,admReg}). It is readily established that we can assume $\alpha=\beta=1$ in the proofs (the arguments for general $\alpha,\beta\in\R_+^*$ are the same).
    \begin{nota}[Regularizer $\reg$]
    \label{standReg}
    We write $\reg=\reg_{1,1}$, that is $\reg(\mu)=\omega(\mu)+\essvar(\mu)$ for every $\mu\in\adm$.
    \end{nota}
    \begin{lem}[Compactness in $(\adm,\reg_{\alpha,\beta})$]
    \label{compactness}
    Let $(\mu^j)\subset\adm$ with $\sup_j\reg_{\alpha,\beta}(\mu^j)<\infty$. Then, up to a subsequence, we have $\mu^j\ws\mu$ a.e.\:and $\reg_{\alpha,\beta}(\mu)\leq\liminf_j\reg_{\alpha,\beta}(\mu^j)$ for some $\mu\in\adm$.
    \end{lem}
    \begin{proof}
    Since $\reg(\mu^j)=\omega(\mu^j)+\essvar(\mu^j)$ is uniformly bounded, we get $\omega(\mu^j)\to\omega\in\R_+$ up to a subsequence. Recall that $\mu^j=\omega(\mu^j)\rho^j$ with $\rho^j\in BV(I;\wass)$.
    
	\underline{Case $\omega>0$:} If $\omega>0$, then $\sup_j\essvar(\rho^j)=\sup_j\omega(\mu^j)^{-1}\essvar(\mu^j)$ must be finite (neglect the finite number of indices for which $\omega(\mu^j)=0$). Moreover, we clearly have $\sup_j\int_IW_1(\rho^j,\delta_0)\e\mathcal{L}<\infty$ because $\C$ is bounded. Further, metric space $\wass$ is separable and every bounded and closed set in $\wass$ is compact by the Banach\textendash Alaoglu theorem ($W_1$ metrizes weak-$*$ convergence). Hence, by \cite[Thm.\:2.4 (i)]{A90} there is some $\rho\in BV(I;\wass)$ such that (up to a subsequence) $\lim_jW_1(\rho^j,\rho)=0$ a.e. and $\essvar(\rho)\leq\liminf_j\essvar(\rho^j)$. Therefore, we obtain $\rho^j\ws\rho$ a.e., thus $\mu^j\ws\mu=\omega\rho$ a.e. In summary, we have
	\begin{equation*}
	\reg(\mu)=\omega(\mu)\reg(\rho)\leq\omega(\mu)\liminf_j\reg(\rho^j)=\liminf_j\omega(\mu^j)\reg(\rho^j)=\liminf_j\reg(\mu^j).
	\end{equation*} 
	\underline{Case $\omega=0$:} If $\omega=0$, then we directly obtain $\mu^j\ws 0$ a.e. Noting $\reg (0)=0$ yields the inequality.
    \end{proof}
	\section{Existence of sparse solution for BV curve tracking}
	\label{Proofs} 
    In this section, we always take $\alpha=\beta=1$ to make the proofs more readable (recall \cref{standReg}). The arguments are the same for general $\alpha,\beta\in\R_+^*$ (as we already noticed in \cref{compactness}).
    First, we prove the existence of a minimizer for \cref{BVtp} (\cref{existence}) and the stability property in \cref{stability}. The first statement is a direct consequence of \cref{compactness}.
    \begin{proof}[Proof of \cref{existence}]
	Recall that $\fid$ is bounded from below and proper. Therefore, we can pick a minimizing sequence $(\mu^j)\subset\mathcal{A}$. Since $\reg(\mu^j)$ is uniformly bounded, we can apply \cref{compactness}: there exists $\mu\in\adm$ such that $\mu^j\ws\mu$ a.e.\:and $\reg(\mu)\leq\liminf_j\reg(\mu^j)$ (up to a subsequence). By the regularity of $\for$ (it satisfies $\for(\mu^j)\rightharpoonup\for(\mu)$) and $\fid$ (it is weakly lower semicontinuous) we obtain $\fid(\for(\mu))\leq\liminf_j\fid(\for(\mu^j))$. Hence, the BV curve $\mu$ is a minimizer of \cref{BVtp}.
	\end{proof}
    \begin{proof}[Proof of the stability statement in \cref{stability}]
    The argument is similar to \cite[Thm.\:4.7]{BF20}. By the triangle inequality, the optimality of the $\mu^j$, $\reg(0)=0$, and $f^j\to \hat f$ we have
	\begin{multline*}
	\fid_{\hat{f}}(\for(\mu^j))+\reg(\mu^j)=\frac{1}{2}\| \for(\mu^j)-\hat f \|_Y^2+\reg(\mu^j)\leq \| \for(\mu^j)-f^j \|_Y^2+\reg(\mu^j)+\| f^j-\hat f \|_Y^2\\
    \leq 2\left(\frac{1}{2}\| \for(\mu^j)-f^j \|_Y^2+\reg(\mu^j)\right)+\| f^j-\hat f \|_Y^2\leq \| \for(0)-f^j \|_Y^2+\| f^j-\hat f \|_Y^2 \leq C<\infty
	\end{multline*}
    for some constant $C>0$. Hence, by the proof of \cref{existence} there exists $\hat{\mu}\in\adm$ with $\mu^j\ws\hat{\mu}$ a.e.\:(up to a subsequence) and $\fid_{\hat{f}}(\for(\hat{\mu}))+\reg(\hat{\mu})\leq \liminf_j\fid_{\hat{f}}(\for(\mu^j))+\reg(\mu^j)$. Further, note that (by taking the $\liminf$ on both sides of the first inequality in the above estimate and using $f^j\to \hat{f}$ again) $\liminf_j\fid_{\hat{f}}(\for(\mu^j))+\reg(\mu^j)\leq\liminf_j\fid_{f^j}(\for(\mu^j))+\reg(\mu^j)$. Using this, the optimality of the $\mu^j$, and the definition of $\fid_f$ (together with $f^j\to \hat f$) we get
	 \begin{equation*}
	 \fid_{\hat{f}}(\for(\hat{\mu}))+\reg(\hat{\mu})\leq\liminf_j\fid_{f^j}(\for(\mu^j))+\reg(\mu^j)\leq\liminf_j\fid_{f^j}(\for(\mu))+\reg(\mu)=\fid_{\hat{f}}(\for(\mu))+\reg(\mu)
	 \end{equation*}
	 for all $\mu\in\adm$, which shows the optimality of $\hat{\mu}$.
    \end{proof}
	\subsection{Characterization of extremal points in $L_c^-(\reg)$}
	\label{extremalpoints}	
	In this section, we characterize the extremal points of the sublevel sets (\cref{extremalpointschar})
	\begin{align*}
	L_c^-(\reg)=\{ \mu\in\adm\:|\:\reg(\mu)\leq c \}.
	\end{align*}
    We will need the following statement.
		\begin{lem}[Bound for essential variation of pushforward under evaluation map]
		\label{varineq}
		Let $\Gamma\in\mathcal{B}(\mathcal{D}_E(\C))$ and $\eta\in \R_+\mathcal{P}(\DE(\C)\cap\{ \var<\infty \} )$ with $\int_{\mathcal{D}_E(\C)}|\J\gamma|(I)\e\eta(\gamma)<\infty$. Assume that $\eta(\Gamma)\in\R_+^*$.
		Then the map $t\mapsto\mu_t=(e_t)_\# (\eta\mres \Gamma)$ satisfies 
		\begin{equation*}
		\essvar(\mu)\leq\int_{\Gamma}\var(\gamma)\:\e\eta(\gamma).
		\end{equation*}
	\end{lem}
	\begin{proof}
		First, we remark that $\int_{\Gamma}\var(\gamma)\:\e\eta(\gamma)$ is well-defined because $\gamma\mapsto\textup{var}(\gamma)$ is lower semicontinuous \cite[Lem.\:2.13]{ALS24}\footnote{\label{foot1}It is also finite. Indeed, we have $\var(\gamma)=\var(\gamma;(0,1))+\gamma_1^-=|\mathrm D\gamma|(0,1)+\gamma_1^-$ \cite[Lem.\:2.5 (1)]{ALS24} and $\gamma_1^-=\lim_{t\nearrow 1}|\gamma(t)-\gamma(1)|\leq\textup{diam}(\Omega)$ ($\gamma_1^-$ exists because $\gamma$ is \cad). Thus $\int_\Gamma\var(\gamma)\mathrm d\eta(\gamma)\leq\int_{\DE(\C)}|\mathrm D\gamma|(I)\mathrm d\eta(\gamma)+\textup{diam}(\Omega)<\infty$ by assumption.} (which also implies that $\DE(\C)\cap\{\textup{var}<\infty\}$ is Borel measurable). By \cite[Thm.\:3.1]{ALS24} the integral $\int_{\mathcal{D}_E(\C)}|\J\gamma|(I)\e\eta(\gamma)$ is well-defined, $\mu \in \eta(\Gamma)\mathcal{D}_W(\C)$, and
		\begin{equation}
		\label[ineq]{in1:eqD}
		|\J\mu|\leq\int_{\Gamma}|\J\gamma|\e\eta(\gamma).
		\end{equation}
		Using the definition of $\essvar$ and the characterizing properties of \cad\:curves we get
		\begin{equation*}
		\essvar(\mu)\leq\var(\mu)=\var(\mu;(0,1))+\mu_1^-,
		\end{equation*}
		where we use the abbreviation (for the jump at $t=1$)
		\begin{equation*}
		\mu_1^-=\lim_{t\nearrow 1}W_1(\mu_t,\mu_1).
		\end{equation*}
		For each $t\in I$ we define a transport plan (as in \cite[Def.\:3.4.9]{B05}) $\pi_t\in\eta(\Gamma)\mathcal{P}(\C\times\C)$ by
		\begin{equation*}
		\langle \varphi,\pi_t\rangle=\int_{\Gamma}\varphi(\gamma(t),\gamma(1))\mathrm{d}\eta(\gamma)
		\end{equation*}
		for all $\varphi\in C(\C\times \C)$ (the integrand $\gamma\mapsto \varphi(\gamma(t),\gamma(1))$ is Borel measurable by \cite[Prop.\:2.15]{ALS24} for all $\varphi\in C(\C\times \C)$). Moreover, the transport plans $\pi_t$ are admissible for $W_1(\mu_t,\mu_1)$ since
		\begin{equation*}
		\pi_t(B\times\C)=(\eta\mres\Gamma)(e_t^{-1}(B))=\mu_t(B)\qquad\textup{and}\qquad\pi_t(\C\times B)=\mu_1(B)
		\end{equation*}
		for all $B\in\mathcal{B}(\C)$. Hence, we obtain
		\begin{equation*}
		\mu_1^-\leq \limsup_{t\nearrow 1}\int_{\C\times\C}|x-y|\mathrm{d}\pi_t(x,y)=\limsup_{t\nearrow 1}\int_\Gamma|\gamma(t)-\gamma(1)|\mathrm{d}\eta(\gamma),
		\end{equation*}
		where we used $\varphi(x,y)=|x-y|$ in the equality. Note that (as in \cref{foot1}) $\gamma\mapsto|\gamma(t)-\gamma(1)|$ is bounded (uniformly in $t\in I$) by $\gamma\mapsto\textup{diam}(\Omega)$. Thus, using dominated convergence,
		\begin{equation*}
		\limsup_{t\nearrow 1}\int_\Gamma|\gamma(t)-\gamma(1)|\mathrm{d}\eta(\gamma)=\int_\Gamma\gamma_1^-\mathrm{d}\eta(\gamma).
		\end{equation*}
		In summary, we obtain (invoking \cite[Lem.\:2.5 (1)]{ALS24} in the first equality and \cref{in1:eqD} plus the above estimate in the second inequality)
		\begin{multline*}
		\essvar(\mu)\leq\var(\mu;(0,1))+\mu_1^-=|\J\mu|((0,1))+\mu_1^-\\
		\leq\int_{\Gamma}|\J\gamma|((0,1))\mathrm{d}\eta(\gamma)+\int_\Gamma\gamma_1^-\mathrm{d}\eta(\gamma)
		=\int_\Gamma\var(\gamma)\mathrm{d}\eta(\gamma).\qedhere
		\end{multline*}
	\end{proof}
	\begin{proof}[Proof of \cref{extremalpointschar}]
        Fix $c\in\R_+$. First, we note that zero is always an extremal point of $L_c^-(\reg)$ because $\reg(0)=\reg(\lambda\mu+(1-\lambda)\nu)$ with $\lambda\in(0,1),\mu,\nu\in L_c^-(\reg)$ implies $\mu=\nu=0$ by $\mu,\nu\geq 0$ and definition of $\reg$. We abbreviate
		\begin{equation*}
		\E=\{ \mu\in\adm\:|\:\reg(\mu)\in\{ 0,c \},\bar{\mu}_t=\omega(\mu)\delta_{x_t}\textup{ for all }t\in[0,1) \}.
		\end{equation*}
		\underline{$\ext(L_c^-(\reg))\subset \E$:} 
		Let $\mu\in\ext(L_c^-(\reg))$. By the above we can assume $\mu\neq 0$, thus $\mathcal{R}(\mu)= c$. Indeed, if we had $\mathcal{R}(\mu)\in (0,c)$, then $\mu=\lambda\cdot 0+(1-\lambda)(1-\lambda)^{-1}\mu$ and $\mathcal{R}((1-\lambda)^{-1}\mu)=(1-\lambda)^{-1}\mathcal{R}(\mu)\leq c$ for $\lambda\in (0,1)$ sufficiently small, thus $0,(1-\lambda)^{-1}\mu\in L_c^-(\reg)$ which would imply that $\mu$ is not extremal. 
		Let $\bar \mu$ be a \cad\:representative of $\mu$ with $\var(\bar{\mu})=\essvar(\mu)$ (which is finite by $\reg(\mu)<\infty$). By \cite[Thm.\:3.3]{ALS24} there exists a measure $\eta\in\omega(\mu)\mathcal{P}(\mathcal{D}_E(\C))$ (concentrated on $\mathcal{D}_E(\C)\cap\{ \var<\infty \}$) with $(e_t)_\#\eta=\bar \mu_t$ for all $t\in I$ such that
		\begin{equation*}
		|\J \bar \mu|=\int_{\mathcal{D}_E(\C)}|\J\gamma|\mathrm{d}\eta(\gamma)
		\end{equation*}
		as well as (see footnote on \cite[p.\:18]{ALS24})
		\begin{equation}
		\label{varint}
		\var(\bar \mu)=\int_{\mathcal{D}_E(\C)}\var(\gamma)\mathrm{d}\eta(\gamma).
		\end{equation}
		We want to show that $\bar \mu_t = \omega(\mu)\delta_{x_t}$ for every $t\in [0,1)$. To this end, we use a similar approach as in the proof of the claim in \cite[Thm.\:6]{bredies2021extremal}. By contradiction, assume that there exist $t_0\in [0,1)$ and $B_0 \in\B(\C)$ with $\bar{\mu}_{t_0}(B_0),\bar{\mu}_{t_0}(\C\setminus B_0)\in \R_+^*$. Now set
		\begin{equation*}
		\Gamma_0=\mathcal{D}_E(\C)\cap e_{t_0}^{-1}(B_0).
		\end{equation*}
		Set $\Gamma_0$ is Borel measurable by \cite[Prop.\:2.15]{ALS24}. Define
		\begin{equation*}
		\lambda_1=\frac{1}{c}\left(\eta(\Gamma_0)+\int_{\Gamma_0}\var(\gamma)\mathrm{d}\eta(\gamma)\right)\quad\textup{and}\quad\lambda_2=\frac{1}{c}\left(\eta(\mathcal{D}_E(\C)\setminus\Gamma_0)+\int_{\mathcal{D}_E(\C)\setminus\Gamma_0}\var(\gamma)\mathrm{d}\eta(\gamma)\right).
		\end{equation*}
		We have $\eta(\Gamma_0)>0$ since
		\begin{equation*}
		\eta(\Gamma_0)=\eta(e_{t_0}^{-1}(B_0))=\bar{\mu}_{t_0}(B_0)>0.
		\end{equation*}
		Therefore, we get $\lambda_1>0$. Similarly, we obtain $\lambda_2>0$. Moreover, 
		\begin{equation*}
		c(\lambda_1+\lambda_2)=\eta( \mathcal{D}_E(\C))+\int_{\mathcal{D}_E(\C)}\var(\gamma)\mathrm{d}\eta(\gamma)=\omega(\mu)+\essvar(\mu)=\reg(\mu)=c
		\end{equation*}
		by \cref{varint} and the choice of $\bar{\mu}$. Therefore, we get $\lambda_1,\lambda_2\in (0,1)$ with $\lambda_1+\lambda_2=1$. Now define $\bar{\mu}^1_t,\bar{\mu}^2_t\in\R_+^*\mathcal{P}(\C)$ by
		\begin{equation*}
		\bar{\mu}_t^1=\lambda_1^{-1}(e_t)_\#(\eta\mres\Gamma_0)\qquad\textup{and}\qquad\bar{\mu}_t^2=\lambda_2^{-1}(e_t)_\#(\eta\mres(\mathcal{D}_E(\C)\setminus\Gamma_0))
		\end{equation*}
		for all $t\in I$. Note that we directly get
		\begin{equation*}
		\bar{\mu}=\lambda_1\bar{\mu}^1+\lambda_2\bar{\mu}^2.
		\end{equation*}
		Now we show that $\bar{\mu}^1,\bar{\mu}^2$ are representatives of $\mu^1,\mu^2\in L_c^-(\reg)$ with $\mu^1\neq\mu^2$ (which yields the desired contradiction). First, note that $\mu^1\in L_c^-(\reg)$ by (the argument for $\mu^2$ is similar)
		\begin{equation*}
		\reg(\mu^1)=\omega(\mu^1)+\essvar(\mu^1)=\lambda_1^{-1}( \eta(\Gamma_0) +\essvar(t\mapsto (e_t)_\#(\eta\mres\Gamma_0)) )\leq c,
		\end{equation*}
		where the inequality follows from the definition of $\lambda_1$ and \cref{varineq}.
		Finally, we show $\mu^1\neq\mu^2$. We have 
		\begin{equation*}
		\lambda_1\bar\mu_{t_0}^1(B_0)=(e_{t_0})_\# (\eta\mres\Gamma_0)(B_0)=\eta(\Gamma_0)>0\quad\textup{and}\quad \lambda_2\bar\mu_{t_0}^2(B_0)=(e_{t_0})_\# (\eta\mres(\mathcal{D}_E(\C)\setminus\Gamma_0))(B_0)=0,
		\end{equation*}
		thus $\bar\mu^1_{t_0}(B_0)\neq\bar\mu^2_{t_0}(B_0)$.
		We prove that there exists $t_1\in (t_0,1)$ with $\bar\mu_t^1\neq\bar\mu_t^2$ for all $t\in(t_0,t_1)$. For a contradiction, assume that there is a sequence $(s_i)\subset (t_0,1)$ with $s_i\searrow t_0$ and $\bar\mu_{s_i}^1=\bar\mu_{s_i}^2$ for all $i$. Since $\bar{\mu}^1,\bar{\mu}^2$ are \cad\:by \cite[Thm.\:3.1]{ALS24}, we obtain
		\begin{equation*}
		0=\lim_iW_1(\bar\mu_{t_0}^1,\bar\mu_{s_i}^1)=\lim_iW_1(\bar\mu_{t_0}^1,\bar\mu_{s_i}^2),
		\end{equation*}
		thus $\bar\mu_{t_0}^1=\bar\mu_{t_0}^2$ (which contradicts $\bar\mu^1_{t_0}(B_0)\neq\bar\mu^2_{t_0}(B_0)$).
        
		\underline{$\ext(L_c^-(\reg))\supset \E$:} 
		Let $\mu\in\mathcal{E}$. The case $\mathcal{R}(\mu)=0$ is trivial. Hence, we can assume $\mathcal{R}(\mu)=c$. Write $\mu=\lambda\mu^1+(1-\lambda)\mu^2$ with $\lambda\in (0,1)$ and $\mu^1,\mu^2\in L_c^-(\reg)$. We show $\mu^1=\mu^2=\mu$. We have
		\begin{equation*}
		c=\mathcal{R}(\mu)\leq\lambda\mathcal{R}(\mu^1)+(1-\lambda)\mathcal{R}(\mu^2)\leq\lambda c+(1-\lambda)c=c.
		\end{equation*}
		Therefore, we get $\mathcal{R}(\mu^1)=\mathcal{R}(\mu^2)=c$. 
		Write $\mu^1 = \omega(\mu^1)\rho^1$ and $\mu^2 = \omega(\mu^2)\rho^2$ with $\rho^1,\rho^2 \in BV(I, \mathcal{W}_1(\C))$. Let $\bar{\mu},\bar \rho^1,\bar \rho^2$ be \cad\:representatives of $\mu,\rho^1,\rho^2$.   
		For every $t\in [0,1)$ and every $B\in\mathcal{B}(\C)$ we have (using $\bar{\mu}_t=\omega(\mu)\delta_{x_t}$)
		\begin{equation*}
		\lambda \omega(\mu^1)\bar{\rho}_t^1(B)+(1-\lambda)\omega(\mu^2)\bar{\rho}_t^2(B)=\bar{\mu}_t(B)=\begin{cases*}
		\omega(\mu)&if $x_t\in B$,\\
		0&else.
		\end{cases*}
		\end{equation*}
		In particular, we obtain
		\begin{equation*}
		\bar{\rho}_t^1(\C\setminus \{ x_t \})=\bar{\rho}_t^2(\C\setminus \{ x_t \})=0,
		\end{equation*}
		thus $\bar \rho_t^1 = \bar \rho_t^2 = \delta_{x_t}$ for every $t \in [0,1)$. 
		Using $\mathcal{R}(\mu^1)=\mathcal{R}(\mu^2)=c$, we finally conclude $\omega(\mu^1)=\omega(\mu^2)=\omega(\mu)$ as we wanted to prove.
	\end{proof}
	\subsection{Application of a representer theorem}
    \label{representersection}
    In this section, we apply \cref{extremalpointschar} to prove \cref{representer}. We will use the following abbreviations (recall the definitions in \cref{BasNot}).
    \begin{nota}[Solution set $\sol$, weak-$*$ topology $\T$]
    \label{someNotation}
    We define $\mathcal{S}$ as the set of solutions of \cref{BVtp} (recall \cref{existence}). Further, we write $\T$ for the weak-$*$ topology on $\V$.
    \end{nota}
    We will show (\cref{extNon}) that $\ext(\mathcal{S})$ is non-empty. This will allow us to pick and characterize an arbitrary element of $\ext(\sol)$ in the proof of \cref{representer}. In order to show $\ext(\mathcal{S})\neq\emptyset$, we use \cref{furtherProps} (see below) and the properties of $\adm$ stated in \cref{structAdm}. 
    \begin{proof}[Proof of \cref{structAdm}]
    First, we argue that $\adm$ is a subset of $\V$. This actually follows directly from the fact that $W_1$ induces the weak-$*$ topology on $\probs(\C)$ \cite[Thm.\:6.9]{V09}. Now it is straightforward to check that the weak-$*$ topology coincides with the weak (subspace) topology on $\probs(\C)$ stemming from $C(\Omega)'$. Hence, any map $f:I\to\probs(\C)$ is Lebesgue measurable (relative to the weak-$*$ topology on $\probs(\C)$) if and only if it is weakly measurable. Now let $\mu=\omega(\mu)\rho\in\adm$. Then, by definition, any representative of $\rho$ is Lebesgue measurable. Further, we clearly have $\|\rho\|_{L_w^{2}}=1$. Thus, using also the above equivalence of the measurability notions, we obtain $\rho\in \V$. Further, since $\V$ is a vector space, we also have $\mu=\omega(\mu)\rho\in \V$, which proves $\adm\subset \V$. Obviously, set $\adm$ is a cone. Its convexity follows from 
        \begin{multline*}
            \lambda \mu^1 + (1-\lambda)\mu^2 \\ = (\lambda\omega(\mu^1) + (1-\lambda)\omega(\mu^2)) \left[\frac{\lambda\omega(\mu^1)}{\lambda\omega(\mu^1) + (1-\lambda)\omega(\mu^2) }\rho^1 + \frac{(1-\lambda)\omega(\mu^2)}{\lambda\omega(\mu^1) + (1-\lambda)\omega(\mu^2) } \rho^2\right] \in \adm
        \end{multline*}
        for all $\mu^1=\omega(\mu^1)\rho^1,\mu^2=\omega(\mu^2)\rho^2\in\adm\setminus\{0\}$ and $\lambda\in(0,1)$. We note that the convex combination in the angular brackets satisfies the required measurability (by the above) and its distance to $\delta_0$ (with respect to $W_1$) is finite because $\Omega$ is bounded.
    \end{proof}
    The next lemma will be used in the proof of \cref{extNon} and \cref{AlgConv}.
    \begin{lem}[Properties of $\reg$ with respect to $\T$]
    \label{furtherProps}
    The regularizer $\reg$ is lower semicontinuous with respect to $\adm\cap\T$. Further, for every $c\in\R_+$ the sublevel set $L_c^-(\reg)\subset\V$ is compact relative to $\T$. In particular, if $(\mu^j)\subset L_c^-(\reg)$, then, up to a subsequence, we have $\mu^j\ws\mu$ (in $\T$) and $\mu^j\ws\mu$ a.e.
    \end{lem}
    \begin{proof}
    Let $(\mu^j)\subset\adm$ be an arbitrary sequence with $\mu^j\ws\mu\in\adm$. We can assume that $\liminf_j\reg(\mu^j)<\infty$ (otherwise there is nothing to show) and restrict to a subsequence with $\liminf_j\reg(\mu^j)=\lim_j\reg(\mu^j)$. By \cref{compactness} we have (if necessary, restricting to a further subsequence) $\mu^j\ws\tilde\mu$ a.e.\:for some $\tilde\mu\in\adm$ with $\reg(\tilde\mu)\leq\liminf_j\reg(\mu^j)$. Therefore, we only need to show that $\mu=\tilde\mu$. Let $\phi\in \Vstar$ be arbitrary. We have $f^j(t)=|\langle\phi_t,\tilde\mu_t-\mu^j_t\rangle|\to 0$ for a.e.\:$t\in I$. Further, we get $0\leq f^j(t)\leq f(t)=\|\phi_t\|_\infty(\omega(\tilde\mu)+\omega(\mu^j))$ for a.e.\:$t\in I$. Note that $f\in L^1(I)$ by the boundedness of $\int_I\|\phi\|_\infty\e\Leb$ (since $I$ is bounded and $\|\phi\|_{L_w^{2}}<\infty$) and $\omega(\mu^j)$ (because $\omega(\mu^j)\leq\reg(\mu^j)$). Thus, by Lebesgue's dominated convergence theorem we obtain
    \begin{equation*}
    \lim_j\langle\phi,\tilde\mu-\mu^j\rangle_{L_w^2}=\lim_j\int_I\langle\phi,\tilde\mu-\mu^j\rangle\e\Leb=\lim_j\int_If^j(t)\e\Leb=\int_I\lim_j f^j(t)\e\Leb=0,
    \end{equation*}
    which (by the uniqueness of weak-$*$ limits) implies $\tilde\mu=\mu$. This proves the lower semicontinuity of $\reg$ with respect to $\adm\cap\T$. Now let $(\nu^j)\subset L_c^-(\reg)$. Since then $\reg(\nu^j)$ is bounded by $c$, the weights $\omega(\nu^j)$ are also bounded. Using $\|\nu^j\|_{L_w^{2}}=\omega(\nu^j)$ and the Banach\textendash Alaoglu theorem, we obtain $\nu^j\ws\nu\in \V$ up to a subsequence. By the same argument as above we actually have (if necessary, restricting to a further subsequence) $\nu^j\ws\tilde\nu$ a.e.\:for some $\tilde\nu\in\adm$ with $\reg(\tilde\nu)\leq\liminf_j\reg(\nu^j)\leq c$ and $\nu=\tilde\nu\in\adm$, which finalizes the proof.
    \end{proof}
    \begin{examp}[Convergence in $L_c^-(\reg)$]
    \label{convExamp}
    In the proof of \cref{furtherProps}, we have seen that $(\mu^j)\subset L_c^-(\reg)$ with $\mu^j\ws\mu\in L_c^-(\reg)$ (with respect to $\T$) implies $\mu^j\ws\mu$ a.e.\:up to a subsequence. This is, in general, not true (in $L_c^-(\reg)$) for the whole sequence as the following example illustrates: For $i$ and $j\in\{ 0,1,\ldots,2^{i-1}-1 \}$ define indices and (spatial and time) intervals by
    \begin{equation*}
        k=k(i,j)=2^{i-1}+j\qquad\textup{and}\qquad \C_k=I_k=\begin{cases*}
            2^{1-i}[j,j+1)&\textup{if} $j\neq2^{i-1}-1 $,\\
            2^{1-i}[j,j+1]&\textup{if} $j=2^{i-1}-1 $.
        \end{cases*}
    \end{equation*}
    For each $i$ the intervals $\Omega_k$ partition $\Omega=[0,1]$ and $\Leb(\Omega_k)=2^{1-i}$. Now define (see \cref{FigBV})
    \begin{equation*}
        \mu_t^k=\begin{cases*}
            2^{i-1}\Leb\mres\Omega_k&\textup{if} $t\in I_k$,\\
            \delta_0&\textup{if} $t\in I\setminus I_k$.
        \end{cases*}
    \end{equation*}
    We get $\omega(\mu^k)=1$ and $\essvar(\mu^k)\leq 2W_1(\delta_0,2^{i-1}\Leb\mres\Omega_k)\leq 2$, thus $\reg(\mu^k)\leq 3$. In particular, we obtain $\mu^k\in L_3^-(\reg)$. Next, we show that $\mu^k\ws\mu$ in $\V$. Let $\phi:I\to (C(\Omega),\|\cdot\|_\infty)$ be continuous. Then
    \begin{equation}
        \label{innerprod}
        \langle\phi,\mu^k\rangle_{L_w^2}=2^{i-1}\int_{I_k}\int_{\C_k}\phi_t(x)\e\Leb(x)\e\Leb(t)+\int_{I\setminus I_k}\phi_t(0)\e\Leb(t).
    \end{equation}
    We have 
    \begin{equation*}
        2^{i-1}\int_{I_k}\int_{\C_k}\phi_t(x)\e\Leb(x)\e\Leb(t)\leq 2^{i-1}\max_{t\in I}\|\phi_t\|_\infty\int_{I_k}\int_{\C_k}1\e\Leb(x)\e\Leb(t)=2^{1-i}\max_{t\in I}\|\phi_t\|_\infty\to 0
    \end{equation*}
    for $i\to\infty$. Further, we can estimate
    \begin{equation*}
        \left|\int_{I\setminus I_k}\phi_t(0)\e\Leb(t)-\int_{I}\phi_t(0)\e\Leb(t)\right|\leq\int_{I_k}|\phi_t(0)|\e\Leb(t)\leq 2^{1-i}\max_{t\in I}\|\phi_t\|_\infty\to 0.
    \end{equation*}
    This yields that the right-hand side in \cref{innerprod} converges to
    \begin{equation*}
        \langle\phi,\mu\equiv\delta_0\rangle_{L_w^2}=\int_{I}\phi_t(0)\e\Leb(t).
    \end{equation*}
    By density of continuous function in the Bochner space $\Vstar$, we have $\mu^k\ws\mu$ in $\V$\footnote{If $\varepsilon>0$, $\phi\in\Vstar$, and $\phi_\varepsilon\in\Vstar$ continuous with $\|\phi-\phi_\varepsilon\|_{L_w^2}<\varepsilon$, then $\langle\phi,\mu^k\rangle_{L_w^2}=\langle\phi-\phi_\varepsilon,\mu^k\rangle_{L_w^2}+\langle\phi_\varepsilon,\mu^k\rangle_{L_w^2}$ and $|\langle\phi-\phi_\varepsilon,\mu^k\rangle_{L_w^2}|\leq\varepsilon\|\mu^k\|_{L_w^2}=\varepsilon$. Thus, we get $\langle\phi,\mu^k\rangle_{L_w^2}\to \langle\phi,\mu\rangle_{L_w^2}$ for $k\to\infty$.}. On the other hand, for all $t\in I$ and $\psi\in C(\Omega)$ we get
    \begin{equation*}
        \langle\psi,\mu_t^k\rangle=\begin{cases*}
            2^{i-1}\int_{\C_k}\psi\e\Leb&\textup{if} $t\in I_k$,\\
            \psi(0)&\textup{if} $t\in I\setminus I_k$.
        \end{cases*}
    \end{equation*}
    The right-hand side does, in general, not converge to $\langle\psi,\mu_t\rangle=\psi(0)$ for a.e.\:$t\in I$. For example, we can choose $\psi(x)=\min(2x,1)$ which yields $2^{i-1}\int_{\C_k}\psi\e\Leb=1$ for countably many $k$ and $\psi(0)=0$. Hence, the convergence $\mu^k\to\mu$ a.e.\:is only valid up to a subsequence (e.g.\:subsequence $\mu^{k_\ell}$ with $k_\ell=k(\ell,0)$).
    \end{examp}
    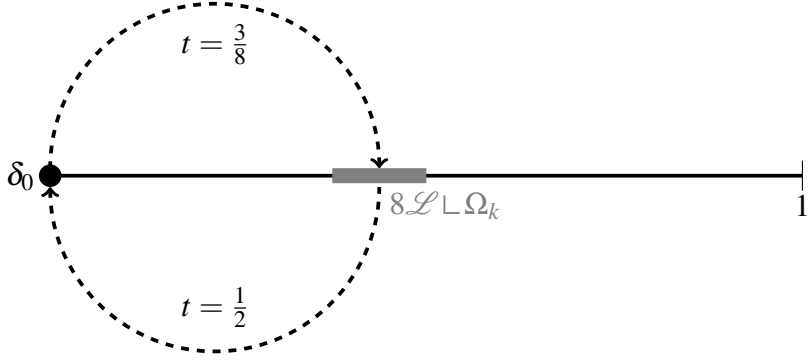
\begin{figure}[t]
    \centering
        \begin{tikzpicture}[scale=10]
        \draw[line width=0.5mm] (0,0) -- (1,0);
        \draw (1,0.02) -- (1,-0.02);
        \draw[gray, line width=2mm] (0.375,0) -- (0.5,0);
        \draw node[fill,circle,minimum size=3mm,inner sep=0pt] {};
        \node at (1,-0.04) {$1$};
        \node at (-0.04,0) {$\delta_0$};
        \node[gray] at (0.525,-0.04) {$8\Leb\mres\Omega_k$};

         \draw[<-,dashed,line width=0.5mm] (0.4375,0.01) arc (0:180:0.21875);
         \draw[<-,dashed,line width=0.5mm] (0,-0.015) arc (180:360:0.21875);

        \node at (0.22,0.175) {$t=\frac{3}{8}$};
         \node at (0.22,-0.175) {$t=\frac{1}{2}$};
        \end{tikzpicture}  
    \caption{Sketch of the (piecewise constant) BV curve $\mu^k$ with $k=k(4,3)$ (\cref{convExamp}).}
    \label{FigBV}
    \end{figure}
    \begin{prop}[$\ext(\sol)\neq\emptyset$]
        \label{extNon}
		Let \cref{linearf} be satisfied (the finiteness of $Y$ is not needed). Then the set $\ext(\mathcal{S})$ is nonempty.
	\end{prop}
    \begin{proof}
        We apply the Krein\textendash Milman theorem. By \cref{structAdm} and $\sol\subset\adm$ we have $\sol\subset \V$. The convexity of $\sol$ follows directly from the convexity of $\adm$ and \cref{linearf}. The compactness of $\sol$ (with respect to $\T$) follows from \cref{furtherProps}: If $(\mu^i)\subset\sol$, then $(\mu^i)\subset L_c^-(\reg)$ for some $c\in\R_+$ by optimality and $\reg\geq0$. Thus, using \cref{furtherProps} and the properties of $\fid,K$, we obtain $\mu^i\ws\mu\in\adm$ and $\mu^i\ws\mu$ a.e.\:up to a subsequence and (by the proof of \cref{existence}) $\mu$ is also optimal, that is $\mu\in\sol$. The Krein\textendash Milman\footnote{The weak-$*$ topology $\T$ is a Hausdorff locally convex vector topology \cite[p.\:68]{Rud2}. Note that the definition of 'vector topology' in \cite[Section 1.6]{Rud2} includes the Hausdorff property (cf.\:\cite[Thm.\:1.12]{Rud2}).} theorem yields that the nonempty, convex, and compact set $\sol\subset \V$ is equal to the convex hull of $\ext(\mathcal{S})$ which, in particular, implies $\ext(\mathcal{S})\neq\emptyset$.
    \end{proof} 
    Next, we show some further properties of the sublevel sets $L_c^-(\reg)$ which we will need in the proof of \cref{representer}.
	\begin{lem}[Properties of $L_c^-(\reg)$]
		\label{propertiesS}
		Let $c\in\R_+$. Then $L_c^-(\reg)$ is linearly closed, does not contain any rays, and its lineality space is trivial, that is $\lin(L_c^-(\reg))=\{ 0 \}$.
	\end{lem}	
	\begin{proof}
        Write $S=L_c^-(\reg)$.
        
		\underline{$S$ linearly closed:} Consider a sequence $\mu^i = \mu^0 + \lambda^i \nu \in S$ where $\lambda^i \in \R$ with $\lambda^i\to\lambda\in\R$ and $\mu^0,\nu\in \V,\nu\neq 0$. We prove $\mu = \mu^0+ \lambda \nu \in S$. Since $\sup_i\reg(\mu^i)\leq c$, \cref{compactness} implies the existence of $\hat{\mu}\in \adm$ with $\mu^i\ws\hat\mu$ a.e.\:(up to a subsequence) and $\reg(\hat{\mu})\leq\liminf_i\reg(\mu^i)$, thus $\hat{\mu}\in S$. Clearly, by definition of $\mu^i$, we also have $\mu^i\ws \mu$ a.e. Therefore, we obtain $\mu=\hat{\mu}\in S$.
        
		\underline{$S$ does not contain any rays:} By contradiction, assume that $S$ contains a ray $R=\mu+\R_+^*\nu$, where $\mu,\nu\in \V,\nu\neq 0$. Define $\mu^k=\mu+k\nu\in R$ for $k\in\mathbb{N}$. Since $\mu^k\in S$, we have $c\geq\reg(\mu^k)=\omega(\mu^k)+\essvar(\mu^k)\geq 0$. Therefore, we get 
        \begin{multline*}       c\geq\omega(\mu^k)=\|\mu^k\|_{L_w^2}=\sup_\phi\langle\phi,\mu^k\rangle_{L_w^2}=\sup_\phi\langle\phi,\mu+k\nu\rangle_{L_w^2}=\sup_\phi\langle\phi,\mu\rangle_{L_w^2}+k\langle\phi,\nu\rangle_{L_w^2}\\
        \geq\sup_\phi-\|\mu\|_{L_w^2}+k\langle\phi,\nu\rangle_{L_w^2}=-\|\mu\|_{L_w^2}+k\|\nu\|_{L_w^2},
        \end{multline*}
        where the supremum is over $\phi\in \Vstar$ with $\|\phi\|_{L_w^2}\leq 1$. Letting $k\to\infty$ yields the contradiction.
        
		\underline{$\lin(S)=\{ 0 \}$:} Note that always $0\in\lin(S)$ by definition. Clearly, the lineality space $\lin(S)$ cannot be nontrivial, since then $\rec(S)$ would be nontrivial and thus $S$ would contain a ray. This was already excluded.
	\end{proof}
	\begin{proof}[Proof of \cref{representer}]
		We apply \cite[Thm.\:1 \&\:Cor.\:2]{BCCDGW19}. Write \cref{BVtp} as 
		\begin{equation*}
		\min_{\mu\in \adm}\fid(\for(\mu))+\reg(\mu)=\min_{\mu\in \V}\fid(\widehat\for(\mu))+\widehat{\reg}(\mu),
		\end{equation*}
        where $\widehat\for$ denotes any linear extension of $K$ to $\V$ (\cref{linearf}) and $\widehat{\mathcal{R}}:\V\to\R\cup\{\infty\}$ is defined by
    \begin{equation*}
    \widehat{\mathcal{R}}(\mu) =\begin{cases*}
		\mathcal{R}(\mu) &if $\mu\in\mathcal{A}$,\\
		\infty&else.
		\end{cases*}
    \end{equation*}
        This problem is convex (\cref{linearf} and \cref{structAdm}). Now let $\mu^{opt}\in\ext(\sol)$ (\cref{extNon}). The assumptions in \cite[Cor.\:2]{BCCDGW19} are satisfied: $\fid$ is convex by \cref{linearf}, the solution set $\sol$ is nonempty by \cref{existence}, and the sublevel set $S=L_c^-(\widehat\reg)=L_c^-(\reg)$ with $c=\reg(\mu^{opt})$ is linearly closed by \cref{propertiesS}. Further, forward operator $\widehat\for$ is linear on $\V$ (cf.\:beginning of \cite[Section 3.1]{BCCDGW19}), regularizer $\widehat{\reg}$ is convex (see setting in \cite[Thm.\:1]{BCCDGW19}) because $\reg$ is ($\adm$ is convex by \cref{structAdm}), and the dimension of the smallest face of $\sol$ at $\mu^{opt}$ is equal to zero since $\mu^{opt}\in\ext(\sol)$. If $\reg(\mu^{opt})>0$, then \cite[Thm.\:1]{BCCDGW19} can be applied by \cite[Cor.\:2]{BCCDGW19} if $\textup{lin}(S)=\{ 0 \}$, which is true by \cref{propertiesS}. Hence, in this case we get that $\mu^{opt}$ is a convex combination of at most $m$ points in $\ext(S)$ or $m-1$ points of which each is in $\ext(S)$ or in an extremal ray of $S$ (which is a ray $R$ whose intersection with open line segments $(\mu^a,\mu^b)$ equals $(\mu^a,\mu^b)$). Since $S$ does not contain any rays (\cref{propertiesS}), we remain in the previous case (because the convex combination is not necessarily strict). If $\reg(\mu^{opt})=0$, then $\mu^{opt}=0$ which is a convex combination of itself. Now we can apply \cref{extremalpointschar},
        \begin{equation*}
            \ext(S)=\ext\left(L_{\reg(\mu^{opt})}^-(\reg)\right)=\{ \mu\in\adm\:|\:\reg(\mu)\in\{ 0,\reg(\mu^{opt}) \},\:\bar{\mu}_t=\omega(\mu)\delta_{x_t}\textup{ for all }t\in[0,1) \}.
        \end{equation*}
        Hence, if $\mu^{opt}=\sum_{i=1}^N\tilde\lambda^i\mu^i$ with $N\leq m$, $\tilde\lambda^i\in[0,1]$, $\sum_{i=1}^N\tilde\lambda^i=1$, and $\mu^i\in\ext(S)$, then $\mu^{opt}_t=\sum_{i=1}^N\lambda^i\delta_{x_t^i}$ for a.e.\:$t\in I$ with $\lambda^i=\tilde\lambda^i\omega(\mu^i)\in\R_+$.
	\end{proof}
    \subsection{Discussion: extension to $\probs(\C)\times\ldots\times\probs(\C)$}
    \label{Discussion}
    One may consider a natural generalization of \cref{BVtp} in which $\adm=\R_+BV(I;\wass)$ is replaced by $\R_+BV(I;\mathcal{W}_H(\C))$, where $\mathcal{W}_H(\C)=(\probs(\Omega)^M,W_H)$ for some metric $W_H$, defined via a norm $H:\R^{M\times n}\to\R_+$, which generalizes $W_1$ (with a slight abuse of notation, we write $W_H=W_1$ if $M=1$ and $H=|\cdot|$); metric $W_H$ will be specified below. In this context, one may interpret $M$ as the number of different labels of moving particles. 
    However, a \textit{natural} superposition principle in the spirit of \cite[Thm.\:3.3]{ALS24} (which was used to characterize the extremal points of $L_c^-(\reg)$ in \cref{extremalpointschar}) does not exist for this setting. In this section, we aim to illustrate this aspect. As a motivation, let us recall the Beckmann formulation of the Wasserstein $1$-distance \cite[Thm.\:4.6]{San} between $\rho^+,\rho^-\in\probs(\C)$,
    \begin{equation*}
    W_1(\rho^+,\rho^-)=\min_T\|T\|_\M,  
    \end{equation*}
    where the minimum is over vector-valued Radon measures $T:\B(\C)\to\R^n$ with distributional divergence equal to $\rho^+-\rho^-$ (equivalently, normal $1$-currents in $\Omega$ whose boundary is equal to $\rho^--\rho^+$). Now let $M>1$ and $\vec\rho^+=(\rho^+_1,\ldots,\rho^+_M)^\intercal,\vec\rho^-=(\rho^-_1,\ldots,\rho^-_M)^\intercal\in\probs(\C)^M$. Fix a norm $H:\R^{M\times n}\to\R_+$. Define
    \begin{equation*}
    W_H(\vec\rho^+,\vec\rho^-)=\min_{ \vec T}\| \vec T\|_H,  
    \end{equation*}
    where the minimum is over matrix-valued Radon measures $\vec T:\B(\C)\to\R^{M\times n}$ with row-wise distributional divergence equal to $\vec\rho^+-\vec\rho^-$ (equivalently, normal $1$-currents in $\Omega$ with coefficients in $\R^M$ whose component-wise boundary is equal to $\vec\rho^--\vec\rho^+$), and $\|\cdot\|_H$ denotes the total variation with respect to $H$, that is $\|\vec T\|_H=\sup \{ H(\vec T(B_1))+H(\vec T(B_2))+\ldots\:|\:B_1,B_2,\ldots\in\B(\C)\textup{ partition of }\C \}$. If $H(\theta\otimes\vec e)=h(\theta)$ for all $\theta\in\R^M$, $\vec e\in\R^n$ with $|\vec e|=1$ for some norm $h:\R^M\to\R_+$, then by \cite[Thm.\:1.10]{LSW322} $W_H(\vec\rho^+,\vec\rho^-)$ can be interpreted as a multi-material transport problem \cite{MMT17,MMSR16} ($h(\theta)$ denotes the cost to transport material vector $\theta$ per unit distance), cf.\:\cite[Ch.\:4.2]{OM23}. Next, we identify an admissible $\vec T$ in $W_H(\vec\rho^+,\vec\rho^-)$ with a \cad\:curve $\mu:I\to\mathcal{W}_H(\C)$. For simplicity, we assume that the divergence-free part of $\vec T$ is equal to zero. We follow the proof of \cite[Prop.\:4.1]{BW}. By \cite[Thm\:C]{S} there exists an $M$-tuple $\vec\eta=(\eta_1,\ldots,\eta_M)^\intercal$ of measures $\eta_i:(\Theta,\B(\Theta))\to[0,\infty)$ ($\Theta$ denoting a space of Lipschitz curves which are identified modulo parameterization with topology induced by a metric $d_\Theta$, cf.\:\cite[Def.\:2.5]{BPSS}) such that 
    \begin{equation*}
    \int_\Omega\Phi:\e\vec T=\int_\Theta\int_I\Phi(\gamma)\dot\gamma\e\Leb\cdot\e\vec\eta(\gamma)
    \end{equation*}
    for all $\Phi\in C(\C;\R^{M\times n})$, where $:$ denotes the Frobenius inner product. By Skorohod’s theorem \cite[Thm.\:6.7]{B} $\vec\eta$ is induced by an $M$-tuple $\vec\chi=(\chi_1,\ldots,\chi_M)^\intercal$ of so-called irrigation patterns $\chi_i:[0,1]\times I\to\C$ between $\rho_i^+$ and $\rho_i^-$ \cite{MSM03,BCM05,MS13}, that is $\chi_i$ is Borel measurable, $\chi_i(p,\cdot)$ is absolutely continuous for $\Leb$-a.e.\:$p\in[0,1]$ ($\chi_i(p,t)$ represents the position of particle $p$ in the reference space $([0,1],\B([0,1]),\Leb\mres[0,1])$ at time $t\in I$), and $\chi_i(\cdot,0)_\#(\Leb\mres[0,1])=\rho_i^+,\chi_i(\cdot,1)_\#(\Leb\mres[0,1])=\rho_i^-$ for all $i=1,\ldots,M$. Now define $\mu:I\to\mathcal{W}_H(\C)$ (note that $\vec\chi(\cdot,t):[0,1]\to\Omega^M$ is Borel measurable) by 
    \begin{equation*}
        \mu_t=\vec\chi(\cdot,t)_\#(\Leb\mres[0,1])
    \end{equation*}
    for every $t\in I$. Then $\mu$ is a \cad\:curve. Indeed, for every $t\in I$, $\varepsilon\in\R\setminus\{ 0\}$ with $t+\varepsilon\in I$, and $\phi\in C(\Omega;\R^M)$ we have
    \begin{equation*}
    \left|\int_\Omega\phi\cdot\e(\mu_{t}-\mu_{t+\varepsilon})\right|\leq\sum_{i=1}^M\int_{[0,1]}|\phi_i(\chi(p,t))-\phi_i(\chi(p,t+\varepsilon))|\e\Leb(p)\to0
    \end{equation*}
    for $\varepsilon\to 0$ by the regularity of $\phi$ and continuity of $t\mapsto\chi(p,t)$ for $\Leb$-a.e.\:$p\in [0,1]$ (Lebesgue's dominated convergence theorem). Thus, we have $\mu_{t+\varepsilon}\ws\mu_{t}$ for $\varepsilon\to 0$, hence $W_H(\mu_{t+\varepsilon},\mu_t)\to0$ for $\varepsilon\to 0$ since $W_H$ metrizes weak-$*$ convergence by norm equivalence. This shows that $\mu:I\to\mathcal{W}_H(\C)$ is \cad. Nevertheless, a natural superposition principle according to \cite[Thm.\:3.3]{ALS24} does not exist which is demonstrated in the following example.
    \begin{examp}[Decompositions of $|\textup{D}\mu|$ for $\mu:I\to\mathcal{W}_H$]
    \label{SuperExamp}
    Let $M=n=2$, $\alpha=\frac{3\pi}{4}$, and $x=\bigl(\begin{smallmatrix}\cos\alpha\\\sin\alpha\end{smallmatrix}\bigl),y=\bigl(\begin{smallmatrix}\cos\alpha\\-\sin\alpha\end{smallmatrix}\bigl),z=\bigl(\begin{smallmatrix}1\\0\end{smallmatrix}\bigl)$. Further assume that $H:\R^{2\times 2}\to\R_+$ satisfies $H(\theta\otimes\vec e)=h(\theta)$ for all $\theta\in\R^2$, $\vec e\in\R^2$ with $|\vec e|=1$ for some norm $h:\R^2\to\R_+$. Define Lipschitz curves $\gamma_1,\gamma_2:I\to\Omega=[-1,1]^2$ (cf.\:\cref{sf:GammaCurves}) by
    \begin{equation*}
        \gamma_1(t)=\begin{cases*}
            (1-2t)x & if $t\in[0,\frac{1}{2})$,\\
            (2t-1)z & else,
        \end{cases*}
        \qquad
        \textup{and}
        \qquad
        \gamma_2(t)=\begin{cases*}
            (1-2t)y & if $t\in[0,\frac{1}{2})$,\\
            (2t-1)z & else.
        \end{cases*}
    \end{equation*}
    Further, let $\mu:I\to\mathcal{W}_H$ be the continuous curve given by 
    \begin{equation*}
        \mu_t=\begin{pmatrix}
            \delta_{\gamma_1(t)}\\
            \delta_{\gamma_2(t)}
        \end{pmatrix}
        \qquad
        \textup{associated with}
        \qquad
        \vec T=\frac{1}{2}\begin{pmatrix}
            \dot\gamma_1^\intercal\mathcal{H}^1\mres(\gamma_1(I))\\
            \dot\gamma_2^\intercal\mathcal{H}^1\mres(\gamma_2(I))
        \end{pmatrix}
    \end{equation*}
    for $t\in I$, where $\mathcal{H}^1$ denotes the one-dimensional Hausdorff measure. Now we may define  
    \begin{equation*}
        \zeta=h(1,0)\delta_{\gamma_1}+h(0,1)\delta_{\gamma_2}\in\R_+\probs(\D_E(\C)).
    \end{equation*}
    Then we have
    \begin{multline*}
        |\textup{D}\mu|((0,1/2))=\textup{var}(\mu,(0,1/2))=h(1,0)\var(\gamma_1,(0,1/2))+h(0,1)\var(\gamma_2,(0,1/2))\\
        =\int_{\D_E(\C)}|\textup{D}\gamma|((0,1/2))\e\zeta(\gamma).
    \end{multline*}
    However, we get
    \begin{multline*}
        |\textup{D}\mu|((1/2,1))=h(1,1)\var(\gamma_1,(1/2,1))\neq \int_{\D_E(\C)}|\textup{D}\gamma|((1/2,1))\e\zeta(\gamma)\quad\textup{if and only if}\\
        h(1,1)<h(1,0)+h(0,1).
    \end{multline*}
    This issue may be solved by considering (see \cref{sf:TildeGammaCurves})
    \begin{equation*}
        \tilde\gamma_1(t)=\begin{cases*}
            \gamma_1(t) & if $t\in[0,\frac{1}{2})$,\\
            \gamma_1(\frac{1}{2}) & else,
        \end{cases*}
        \quad
        \tilde\gamma_2(t)=\begin{cases*}
            \gamma_2(t) & if $t\in[0,\frac{1}{2})$,\\
            \gamma_2(\frac{1}{2}) & else,
        \end{cases*}
        \quad
        \textup{and}
        \quad
        \tilde\gamma_3(t)=\begin{cases*}
           \gamma_1(\frac{1}{2}) & if $t\in[0,\frac{1}{2})$,\\
            \gamma_1(t) & else.
        \end{cases*}
    \end{equation*}
    Define $\tilde\zeta = h(1,0)\delta_{\tilde\gamma_1}+h(0,1)\delta_{\tilde\gamma_2}+h(1,1)\delta_{\tilde\gamma_3}$. Then $|\textup{D}\mu|=\int_{\D_E(\C)}|\textup{D}\tilde\gamma|\e\tilde\zeta(\tilde\gamma)$ but $\tilde\zeta$ is not normalized: $\tilde\zeta\notin\probs(\D_E(\C))$. Further, measure $\tilde\zeta$ is not in a natural sense defined via $\mu$. Another approach would be to consider $\vec\zeta\in\probs(\D_E(\C))^2$ given by
    \begin{equation*}
        \vec\zeta=\begin{pmatrix}
            \delta_{\gamma_1}\\
            \delta_{\gamma_2}
        \end{pmatrix}.
    \end{equation*}
    Then we have
    \begin{multline*}
        \int_{\D_E(\C)}h\left( \frac{\e\vec\zeta}{\e|\vec\zeta|}(\gamma)\right)|\textup{D}\gamma|\e|\vec\zeta|(\gamma)=h(1,0)|\textup{D}\gamma_1|+h(0,1)|\textup{D}\gamma_2|\neq |\textup{D}\mu|\quad\textup{if and only if}\\
        h(1,1)<h(1,0)+h(0,1).
    \end{multline*}
    Finally, we can resolve this by defining $\vec\xi\in\probs(\D_E(\C))^2$ by
    \begin{equation*}
        \vec\xi=\frac{1}{2}\begin{pmatrix}
            \delta_{\tilde\gamma_1}+\delta_{\tilde\gamma_3}\\
            \delta_{\tilde\gamma_2}+\delta_{\tilde\gamma_3}
        \end{pmatrix}.
    \end{equation*}
    This yields 
    \begin{equation*}
        \int_{\D_E(\C)}h\left( \frac{\e\vec\xi}{\e|\vec\xi|}(\gamma)\right)|\textup{D}\gamma|\e|\vec\xi|(\gamma)=h(1,0)|\textup{D}\tilde\gamma_1|+h(0,1)|\textup{D}\tilde\gamma_2|+h(1,1)|\textup{D}\tilde\gamma_2|= |\textup{D}\mu|
    \end{equation*}
    as desired. Measure $\vec\xi$ is still not naturally defined via $\mu$. Nevertheless, it is normalized in the sense that $\vec\xi\in\probs(\Omega)^2$. We believe that the formula (with $\vec\xi$ defined via $\mu$)
    \begin{equation*}
        |\textup{D}\mu|=\int_{\D_E(\C)}h\left( \frac{\e\vec\xi}{\e|\vec\xi|}(\gamma)\right)|\textup{D}\gamma|\e|\vec\xi|(\gamma)
    \end{equation*}
    holds in a general setting.
    \end{examp}
	\begin{figure}[t]
		\centering
		\begin{subfigure}[b]{0.45\textwidth}
			\centering
            \begin{tikzpicture}[scale=3]
            \draw[dashed] (1,0) arc (0:360:1);
            \draw[line width=1mm,lightgray] ({cos(135)},{sin(135)}) -- (0,0);
            \draw[line width=1mm,gray] ({cos(135)},{-sin(135)}) -- (0,0);
            \draw[line width=1mm,lightgray] (1,0.01) -- (0,0.01);
            \draw[line width=1mm,gray] (1,-0.01) -- (0,-0.01);
            \draw[->,line width=1mm,lightgray] ({cos(135)},{sin(135)}) -- ({2*cos(135)/3},{2*sin(135)/3});
            \draw[->,line width=1mm,gray] ({cos(135)},{-sin(135)}) -- ({2*cos(135)/3},{-2*sin(135)/3});

            \node[fill,circle,minimum size=3mm,inner sep=0pt,label=left:{$x$}] at ({cos(135)},{sin(135)}) {};
            \node[fill,circle,minimum size=3mm,inner sep=0pt,label=left:{$y$}] at ({cos(135)},{-sin(135)}) {};
            \node[fill,circle,minimum size=3mm,inner sep=0pt,label=right:{$z$}] at (1,0) {};
            
            \node[label=right:{$\textcolor{lightgray}{\gamma_1}$}] at ({cos(135)/2},{sin(135)/2}) {};
            \node[label=right:{$\textcolor{gray}{\gamma_2}$}] at ({cos(135)/2},{-sin(135)/2}) {};
            \end{tikzpicture}  
			\caption{$z=\gamma_1(1)=\gamma_2(1)$.}
			\label{sf:GammaCurves}
		\end{subfigure}
        \begin{subfigure}[b]{0.45\textwidth}
			\centering			
           \begin{tikzpicture}[scale=3]
            \draw[dashed] (1,0) arc (0:360:1);
            \draw[line width=1mm,lightgray] ({cos(135)},{sin(135)}) -- (0,0);
            \draw[line width=1mm,gray] ({cos(135)},{-sin(135)}) -- (0,0);
            \draw[line width=1mm] (1,0) -- (0,0);
            \draw[->,line width=1mm,lightgray] ({cos(135)},{sin(135)}) -- ({2*cos(135)/3},{2*sin(135)/3});
            \draw[->,line width=1mm,gray] ({cos(135)},{-sin(135)}) -- ({2*cos(135)/3},{-2*sin(135)/3});
            \draw[->,line width=1mm] (0,0) -- ({1/3},{0});

            \node[fill,circle,minimum size=3mm,inner sep=0pt,label=left:{$x$}] at ({cos(135)},{sin(135)}) {};
            \node[fill,circle,minimum size=3mm,inner sep=0pt,label=left:{$y$}] at ({cos(135)},{-sin(135)}) {};
            \node[fill,circle,minimum size=3mm,inner sep=0pt,label=right:{$z$}] at (1,0) {};
            \node[fill,circle,minimum size=3mm,inner sep=0pt] at (0,0) {};

            \node[label=above:{$\tilde\gamma_3$}] at ({1/2},{0}) {};
            \node[label=right:{$\textcolor{lightgray}{\tilde\gamma_1}$}] at ({cos(135)/2},{sin(135)/2}) {};
            \node[label=right:{$\textcolor{gray}{\tilde\gamma_2}$}] at ({cos(135)/2},{-sin(135)/2}) {};
            \end{tikzpicture}  
			\caption{$\{0\}=\tilde\gamma_1([\frac{1}{2},1])=\tilde\gamma_2([\frac{1}{2},1])=\tilde\gamma_3([0,\frac{1}{2}])$.}
			\label{sf:TildeGammaCurves}
		\end{subfigure}
		\caption{Lipschitz curves $\gamma_i$ and $\tilde\gamma_i$ from \cref{SuperExamp}.}
	\end{figure}
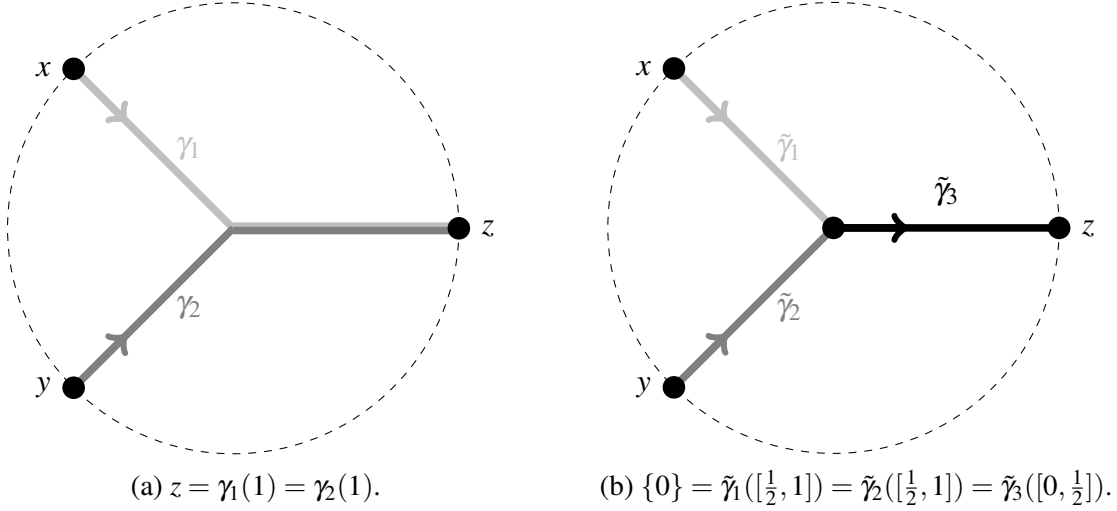


\section{Fully-corrective generalized conditional gradient method for BV curve tracking}
\label{OptAlg}
In this section, we design an algorithm which computes (approximate) solutions of \cref{BVtp} by exploiting the sparse structure given by \cref{representer} and, in particular, using the characterization of extremal points (\cref{extremalpointschar}). The algorithm will be an instance of the (grid-free) fully-corrective generalized conditional gradient (FC-GCG) algorithm introduced in \cite{bredies2024asymptotic} applied to our BV curve tracking problem. Similar algorithms have been used to approximate solutions of dynamic inverse problems regularized with Wasserstein-$2$ transport energies \cite{KCFR22, duval2024dynamical}. 
In order to apply \cite[Alg.\:1]{bredies2024asymptotic}, we require the following (in addition to the assumptions in \cref{BVtp}).
\begin{assump}[Numerical setup]
\label{NumSetup}
The fidelity $\fid$ is strictly convex, $\R$-valued, and Fr\'{e}chet differentiable such that the restriction of $D\fid :Y\to Y$ to any compact subset of $Y$ is Lipschitz. Further, forward operator $K$ is a linear map which can be linearly extended to a weak-$*$-to-strong continuous map on $\V$.
\end{assump}
Note that, compared to \cref{linearf}, data space $Y$ is not necessarily finite-dimensional. In our numerical experiments, we will use $Y=\R^{L\times (M+1)}$ (implying the validity of \cref{representer} under \cref{NumSetup}) and a discretization of an instance of the class of forward operators specified in \cref{KFinVal}. Further, we will use a quadratic fidelity term as in \cref{stability} (whose Fr\'{e}chet derivative at $y_0$ applied to $y$ is just the inner product $(y_0-f,y)_Y$ and therefore satisfies \cref{NumSetup}). 
Moreover, we re-introduce the regularization parameters $\alpha,\beta\in\R_+^*$ to calibrate the (discrete) regularizer in our computations.
\begin{prob}[Auxiliary problem formulation on $\V$]
\label{prob1:def}
Let \cref{NumSetup} be satisfied and $\widehat\for$ any linear extension of $K$ to $\V$ which is weak-$*$-to-strong continuous. Further, define $\widehat{\mathcal{R}}_{\alpha,\beta}:\V\to\R\cup\{\infty\}$ as in the proof of \cref{representer}, 
    \begin{equation*}
    \widehat{\mathcal{R}}_{\alpha,\beta}(\mu) =\begin{cases*}
		\mathcal{R}_{\alpha,\beta}(\mu) &if $\mu\in\mathcal{A}$,\\
		\infty&else.
		\end{cases*}
    \end{equation*}
Then \cref{BVtp} becomes
\begin{equation*}
    \min_{\mu \in \V} \mathcal J_{\alpha,\beta}(\mu), \qquad \textit{where} \qquad  \mathcal J_{\alpha,\beta}(\mu) = \fid(\widehat K\mu) + \widehat\reg_{\alpha,\beta}(\mu).
\end{equation*}
\end{prob}
\begin{rem}[Existence of pre-adjoint of $\widehat K$]
Clearly, forward operator $\widehat K:\V\to Y$ is weak-$*$-to-weak continuous. Thus, it admits a continuous and linear pre-adjoint \cite[Rem.\:3.2]{bredies2013inverse}. A simple application of the Hahn\textendash Banach separation theorem shows that it is unique.
\end{rem}
\begin{nota}[Pre-adjoint of $\widehat K$]
We write $\widehat K_*:Y\to\Vstar$ for the pre-adjoint of $\widehat K$ in \cref{prob1:def}, that is 
\begin{equation*}
\langle\widehat K_*f,\mu\rangle_{L_w^2}=(f,\widehat K\mu)_Y
\end{equation*}
for all $f\in Y,\mu\in\V$.
\end{nota}
\begin{prop}[{Embedding into the setting of \cite[Section 2]{bredies2024asymptotic}}]\label{prop:embedding}
Consider \cref{prob1:def}. Then the assumptions in \cite[Section 2]{bredies2024asymptotic} (including \cite[Section 2, (A1)-(A3)]{bredies2024asymptotic}) are satisfied.
\end{prop}
\begin{proof}
Regularizer $\widehat \reg$ is convex by the convexity of $\adm$ (\cref{structAdm}), lower semicontinuous with compact sublevel sets (with respect to $\T$) by \cref{furtherProps}, and nonnegative homogeneous by definition: $\widehat\reg(\lambda\mu)=\lambda\widehat\reg(\mu)$ for all $\lambda\in\R_+,\mu\in \V$ (with the convention $0\cdot\infty=0$). The other requirements in \cite[Section 2]{bredies2024asymptotic} readily follow from the setup in \cref{prob1:def}.
\end{proof}
\subsection{Description of the algorithm}
\label{DescrAlg}
In this section, we explain the FC-GCG algorithm which we will use to compute (approximate) minimizers of \cref{prob1:def} and relate it to the steps of the general FC-GCG algorithm in \cite{bredies2024asymptotic}.

In the $k$-th iteration, the algorithm updates two \textit{active sets} (term `active' indicates that they are modified at each iteration), namely a set of BV curves in $\C$ denoted by $\Gamma^k = \{\gamma^k_i\}_{i=1}^{N_{k}} \subset BV(I;\C)$ and a set of weights $\Lambda^k = \{\lambda^k_i\}_{i=1}^{N_{k}} \subset \R_+^*$ such that the BV curves $\mu^k \in \adm$ given by
\begin{equation*}
    \mu^k = \sum_{i=1}^{N_{k}} \lambda^k_i \delta_{\gamma_i^k} 
    \end{equation*}
 converge (up to a subsequence, with respect to $\T$) to a solution of \cref{prob1:def} (see \cref{AlgConv}). Note that each iterate $\mu^k$ is sparse but, a priori, the limit (or ground truth) may be non-sparse. 
At each iteration of the algorithm, the sets $\Gamma^k$ and $\Lambda^k$ are updated to $\Gamma^{k+1}$ and $\Lambda^{k+1}$ following several steps which we will describe below. We use the following abbreviation.
 \begin{nota}[Dual variable $p^k$]
    \label{dual}
    In the remainder, we write (often referred to as `dual variable')
\begin{equation*}
p^k = -\widehat K_*D\mathcal{F}(\widehat K\mu^k)\in\Vstar,
\end{equation*}
see, e.g.,\:\cite[Proposition 2.3]{bredies2024asymptotic}.
\end{nota}
The update of the active sets $\Gamma^k$ and $\Lambda^k$ consists of the following three steps:

\underline{1.\:Insertion step:} A new BV curve $\gamma^k_{N_k + 1} \in BV(I;\Omega)$ is determined by solving the following variational problem: 
\begin{equation}
\label[probl]{prob2:insertion}
    \gamma^k_{N_k + 1} \in \argmax_{\gamma \in BV(I;\Omega)}  \max\left\{ a(\gamma)   \int_I p_t^k(\gamma(t)) \e\Leb(t), 0 \right\},
\end{equation}
where 
\begin{align}\label{eq:cons}
a(\gamma) = (\alpha + \beta\essvar(\gamma))^{-1}.    
\end{align}
Problem \eqref{prob2:insertion} corresponds to the insertion step of FC-GCG algorithms defined in \cite[(3.1)]{bredies2024asymptotic},
\begin{align}\label[probl]{eq:insextre}
    \hat u\in\argmax_{u \in \ext (L_1^-(\reg_{\alpha,\beta}))} \langle p^k, u\rangle_{L_w^2}. 
\end{align}
Indeed, due to our characterization of extremal points in \cref{extremalpointschar} (compare with \cite[(8)]{KCFR22}) one can readily check that \cref{eq:insextre} is equivalent to \cref{prob2:insertion}. As a consequence, the $\argmax$ in \cref{prob2:insertion} is non-empty by \cite[Lem.\:A.1]{bredies2024asymptotic}. Note that, in general, functions $p^k_t$ may be non-convex for all $t$ in a set of positive $\Leb$-measure which makes \cref{prob2:insertion} non-convex.
Therefore, its accurate solution is challenging. For details related to the practical implementation of the optimization in \cref{prob2:insertion} we refer to \cref{ImplDet}.
After \cref{prob2:insertion} is solved, the BV curve $\gamma^k_{N_k + 1}$ is added to the active set of curves, $\Gamma^{k,+} = \Gamma^k \cup \{\gamma^k_{N_k + 1}\}$. This justifies the name insertion step.

\underline{2.\:Coefficients optimization step:}
The active set $\Lambda^k$ is updated to $\Lambda^{k,+} = \{\hat \lambda_i^{k}\}_{i=1}^{N_k+1}$ where $\hat \lambda_i^{k}$ are obtained by solving the finite-dimensional problem
\begin{equation}
\label[probl]{prob3:coeff}
    \hat \lambda^{k} = \argmin_{\lambda \in \R_+^{N_k+1}}\mathcal{F}\left(\sum_{i=1}^{N_k + 1} \lambda_i a(\gamma^k_i) \widehat K\delta_{\gamma^k_i} \right) + \sum_{i=1}^{N_k + 1} \lambda_i,
\end{equation}
where $\gamma_1^k,\ldots,\gamma_{N_k}^k\in\Gamma^k$ are transferred from the $(k-1)$-th iteration.
This problem corresponds to the coefficients optimization step of FC-GCG algorithms defined in \cite[(3.4)]{bredies2024asymptotic}. Problem \eqref{prob3:coeff} is easily solvable by proximal methods or Newton methods.

\underline{3.\:Pruning:} We define $N_{k+1}$ as the number of non-zero coefficients in $\hat \lambda^{k}$ and construct $\Gamma^{k+1}$ and $\Lambda^{k+1}$ by removing from $\Gamma^{k,+}$ and $\Lambda^{k,+}$ the elements whose indices correspond to the zero coefficients.

\underline{Stopping criterion:}
We add a stopping criterion which is derived from the optimality conditions associated with \cref{prob1:def}.  
In particular, we terminate the algorithm in iteration $k \in\mathbb{N}$ with output $\mu^k$ if 
\begin{equation}
\label[ineq]{in2:stopping}
    a( \gamma^k_{N_k+1} )\int_I p_t^k(\gamma^k_{N_k+1}(t)) \e\Leb(t)  \leq 1.
\end{equation}
This criterion matches the one proposed in \cite[Prop.\:3.1]{bredies2024asymptotic} where it is formulated as (cf.\:\cite[(3.5)]{bredies2024asymptotic})
\begin{equation*}\label{eq:stooo}
    \max_{u \in \ext (L_1^-(\reg_{\alpha,\beta}))} \langle p^k, u\rangle_{L_w^2} \leq 1.
\end{equation*}
Again, this follows from our characterization of extremal points in \cref{extremalpointschar}.
In our implementation, we will add a tolerance $\varepsilon>0$ on the right-hand side of \cref{in2:stopping}.
\begin{rem}[Insertion of $0\in\adm$]
Note that, if  
\begin{align*}
\max_{\gamma \in BV(I;\Omega)}  \max\left\{a(\gamma)   \int_I p_t^k(\gamma(t)) \e\Leb(t), 0 \right\}  = 0,
\end{align*}
which corresponds to the insertion of the extremal point $\hat u = 0$ in \cref{eq:insextre},
then it necessarily holds that $\max_{\gamma \in BV(I;\Omega)}a(\gamma)   \int_I p_t^k(\gamma(t)) \e\Leb(t) \leq 1$, implying that the stopping criterion in \cref{in2:stopping} is satisfied. Therefore, we can replace \cref{prob2:insertion} with the simplified insertion step 
\begin{equation}\label[probl]{prob:newinsertionstep}
\gamma_{N_{k}+1}^k\in\argmax_{\gamma \in BV(I;\Omega)}  a(\gamma) \int_I p_t^k(\gamma(t)) \e\Leb(t).
\end{equation}
\end{rem}
\Cref{AlgDynamic} summarizes the steps described above.
  \begin{algorithm}[t]
  \caption{Fully-corrective generalized conditional gradient method for BV curve tracking}
  \label{AlgDynamic}
  \begin{algorithmic}
    \REQUIRE $\Lambda^0 = \emptyset$,
$\Gamma^0 = \emptyset$, $N_0 = 0$, $\mu^0 = 0$
\FOR{$k=0,1,2,\dots$}
    \STATE $p^k \gets  -\widehat K_*D\mathcal{F}(\widehat K\mu^k)$
    \STATE $\gamma^k_{N_k+1} \in  \argmax_{\gamma \in BV(I;\Omega)}  a(\gamma) \int_I p_t^k(\gamma(t)) \e\Leb(t)$ 
    \IF{$k \in\mathbb{N}$ and $a(\gamma^k_{N_k+1})\int_I p_t^k(\gamma^k_{N_k+1}(t)) \e\Leb(t)  \leq 1$}
    \RETURN $\mu^k$
    \ENDIF
    \STATE $(\hat \lambda^k_1, \ldots, \hat \lambda^k_{N_k+1})^\intercal\in  \argmin_{\lambda \in \R_+^{N_k+1}}\mathcal{F}\left(\sum_{i=1}^{N_k+1} \lambda_i a(\gamma^k_i)\widehat K\delta_{\gamma^k_i} \right) + \sum_{i=1}^{N_k+1} \lambda_i$ 
    \STATE $\Gamma^{k+1} \gets \left\{\gamma_i^k \in \Gamma^{k} \cup \{\gamma^k_{N_k+1}\}\:\middle|\: \hat \lambda^k_i >0\right\}$, $\Lambda^{k+1} \gets \{\hat \lambda^k_i\:|\: \hat \lambda^k_i >0\}$ 
    \STATE $\mu^{k+1} \gets \sum_{i=1}^{N_{k} + 1} \hat \lambda^k_i \delta_{\gamma_i^k}$
    \STATE $N_{k + 1} \gets \# \Gamma^{k+1}$
\ENDFOR
\end{algorithmic}
\end{algorithm}
\cite[Thm.\:3.3]{bredies2024asymptotic} and \cref{prop:embedding} imply the convergence properties of the iterates produced by \cref{AlgDynamic}. 
We state the rate of convergence in terms of the residuals 
\begin{align}\label{eq:residuals}
    r(\mu^k) = \mathcal J_{\alpha,\beta}(\mu^k) - \min_{\mu \in \adm} \mathcal J_{\alpha,\beta}(\mu).
\end{align}
\begin{coro}[{Sublinear convergence of \cref{AlgDynamic}, cf.\:\cite[Thm.\:3.3]{bredies2024asymptotic}}]
\label{AlgConv}
The iterates $\mu^k \in \adm$ produced by \cref{AlgDynamic} (extended with $\mu^k= \mu^{k_0}$ for $k\geq k_0$ if the algorithm terminates in iteration $k_0$) converge, up to subsequences, in $\T$ to a minimizer of \cref{prob1:def}. Moreover, there exists $C\in\R_+^*$ such that
\begin{equation*}
\label[ineq]{in3:residuall}
    r(\mu^k) \leq \frac{C}{k+1}
\end{equation*}
for all $k \in \mathbb{N}$.
\end{coro}

\subsection{Discretization approach}
\label{DerivationForward}
In our implementation of \cref{AlgDynamic}, we will use discretizations which are obtained by letting the temporal blur at discrete time points go to zero. In this section, we will make this conception rigorous. We will need the following definition.
	\begin{defin}[Nascent $(\theta,t)$-delta functions]
		Let $\theta\in[0,1]$ and $ t\in I$. A family of functions $\phi^\delta\in C(I;\R_+)$ (indexed by $\delta\in\R_+^*$) is called a family of \textbf{nascent $(\theta,t)$-delta functions} if 
		\begin{itemize}
			\item $\phi^\delta(s)=0$ for all $s\in I\setminus[ t-\delta, t+\delta]$,
			\item the numbers $\theta^\delta=\int_{[0, t]}\phi^\delta\e\Leb$ satisfy 
			\begin{equation*}
			\theta^\delta\in[0,1],\qquad\int_{[ t,1]}\phi^\delta\e\Leb=1-\theta^\delta,\qquad\textup{and}\qquad\theta^\delta\to\theta\textup{ for }\delta\to 0.
			\end{equation*}
		\end{itemize}
	\end{defin} 
	Note that necessarily $\theta^\delta\equiv 0$ if $ t=0$ and $\theta^\delta\equiv 1$ if $ t=1$. Thus, for $t=0$ (respectively $t=1$) such a family only exists if $\theta=0$ (respectively $\theta=1$).
	\begin{lem}[`$(\theta,t)$-delta']
		\label{WeightedLemma}
		Let $\psi\in BV(I;\R)$ and $\phi^\delta\in C(I;\R_+)$ a family of nascent $(\theta,t)$-delta functions. Then we have
		\begin{equation*}
		\lim_{\delta\to 0}\int_{[0, t]}\phi^\delta\psi\e\Leb=\theta\bar{\psi}_{ t}^-\qquad\textup{and}\qquad\lim_{\delta\to 0}\int_{[ t,1]}\phi^\delta\psi\e\Leb=(1-\theta)\bar{\psi}_{ t},
		\end{equation*}
		where $\bar{\psi}$ is any c\`{a}dl\`{a}g representative of $\psi$.
	\end{lem}
	\begin{proof}
		We only prove the first equality (the proof of the second is similar). Let $\varepsilon\in\R_+^*$ and $\delta=\delta(\varepsilon)\in\R_+^*$ with $\bar{\psi}_s\in[\bar{\psi}_{ t}^--\varepsilon,\bar{\psi}_{ t}^-+\varepsilon]$ for all $s\in [ t-\delta, t]\cap I$ and $\delta\to 0$ for $\varepsilon\to 0$ (possible because $\bar{\psi}_t^-$ exists). Then, using $\int_{[0, t]}\phi^\delta\e\Leb=\theta^\delta\to\theta$, $\phi^\delta\in\R_+$, and $\phi= 0$ on $[0, t-\delta)\cap I$, we obtain
		\begin{equation*}
		\theta\bar{\psi}_{ t}^-=\lim_{\varepsilon\to 0}\int_{[0, t]}\phi^{\delta}\e\Leb\cdot(\bar{\psi}_{ t}^--\varepsilon)\leq\lim_{\varepsilon\to 0}\int_{[0, t]}\phi^{\delta}\psi\e\Leb\leq\lim_{\varepsilon\to 0} \int_{[0, t]}\phi^{\delta}\e\Leb\cdot(\bar{\psi}_{ t}^-+\varepsilon)=\theta\bar{\psi}_{ t}^-.\qedhere
		\end{equation*}
	\end{proof}  
    The following statement will be applied to our forward operator from \cref{KFinVal}.
    \begin{prop}[Vanishing temporal blur]
    \label{VanishingBlur}
    Let $\mu\in\adm$, $\phi^\delta\in C(I;\R_+)$ a family of nascent $(\theta,t)$-delta functions, and $\Phi:\Omega\to\R$ Lipschitz. Then we get
	\begin{multline*}
	\lim_{\delta\to 0}\int_{[0, t]}\phi^\delta(s)\int_\Omega\Phi(x)\e\mu_s(x)\e\Leb(s)=\theta\int_\Omega\Phi(x)\e\bar\mu_t^-(x)\qquad\textup{and}\\
    \lim_{\delta\to 0}\int_{[t,1]}\phi^\delta(s)\int_\Omega\Phi(x)\e\mu_s(x)\e\Leb(s)=(1-\theta)\int_\Omega\Phi(x)\e\bar\mu_t(x),
	\end{multline*}
	where $\bar{\mu}$ is any c\`{a}dl\`{a}g representative of $\mu$.
    \end{prop}
    \begin{proof}
    As in \cref{WeightedLemma}, we only prove the first equality (the proof of the second is similar). Define $\psi_s=\int_\Omega\Phi\e\mu_s$ for a.e.\:$s\in I$ (recall that $\psi\in L^1(I;\R)$ by $\mu\in\V$). We prove $\psi\in BV(I;\R)$ and apply \cref{WeightedLemma}. We have ($\textup{Lip}(\Phi)$ denoting the Lipschitz constant of $\Phi$, we can assume $\textup{Lip}(\Phi)\in\R_+^*$ because the other case is trivial)
    \begin{multline*}
    \essvar(\psi)\leq\textup{var}\left(s\mapsto\int_\Omega\Phi\e\bar\mu_s\right)=\sup \left\{\sum_{i=1}^N \left|\int_\Omega\Phi\e(\bar\mu_{s_i}-\bar\mu_{s_{i-1}}) \right|\:\middle|\: s_0 < s_1<\ldots < s_N,\, s_i \in I  \right\}\\
    =\textup{Lip}(\Phi)\sup \left\{\sum_{i=1}^N \left|\int_\Omega\frac{1}{\textup{Lip}(\Phi)}\Phi\e(\bar\mu_{s_i}-\bar\mu_{s_{i-1}}) \right|\:\middle|\: s_0 < s_1<\ldots < s_N,\, s_i \in I  \right\}\\
    \leq\textup{Lip}(\Phi)\sup \left\{\sum_{i=1}^N W_1(\bar\mu_{s_i},\bar\mu_{s_{i-1}})\:\middle|\: s_0 < s_1<\ldots < s_N,\, s_i \in I  \right\}=\textup{Lip}(\Phi)\textup{var}(\bar\mu)<\infty
    \end{multline*}
    since $\mu\in\adm$, where we used the Kantorovich\textendash Rubinstein formula in the second inequality. Hence, we obtain $\psi\in BV(I;\R)$. Application of \cref{WeightedLemma} yields
    \begin{equation*}
    \lim_{\delta\to 0}\int_{[0, t]}\phi^\delta(s)\int_\Omega\Phi\e\mu_s\e\Leb(s)=\theta\bar{\psi}_{ t}^-=\theta\int_\Omega\Phi\e\bar\mu_t^-,
    \end{equation*}
    where the last equation follows from the fact that $W_1$ metrizes weak-$*$ convergence and $\Phi\in C(\Omega)$.
    \end{proof}
    Next, let us define a reasonable discretization of a forward operator as in \cref{KFinVal}. To this end, fix a discretization $0=t_0<\ldots<t_M=1$ of the time interval $I=[0,1]$ and, for each $j=0,1,\ldots,M$, let $\phi^\delta_j\in C(I;\R_+)$ be a family of nascent $(\theta_j,t_j)$-delta functions (in particular, we have $\theta_0=0$ and $\theta_M=1$). Define $\widetilde K=\widetilde K(\delta) : L_w^2(I;\mathcal{M}(\Omega)) \to \R^{L\times (M+1)}$ by
\begin{equation}
\label{eq:linex2}
    (\widetilde K\mu)^i_j = \int_I \phi^\delta_j(s)\int_\Omega \Phi^i(x)\e\mu_s(x)\e\Leb(s),
\end{equation}   
where $\Phi^1,\ldots,\Phi^L:\Omega\to\R$ are Lipschitz and may correspond to some spatial points $x_1,\ldots,x_L\in\Omega$, see \cref{KFinVal}. Recall that $\phi_j^\delta$ can be interpreted as the temporal blur at $t_j$ (of order $\delta$). Given $\mu\in\adm$, we may define matrix $K_0\mu\in\R^{L\times (M+1)}$ by (\cref{VanishingBlur})
 \begin{equation*}
     ( K_0\mu)^i_j = \lim_{\delta\to 0}(\widetilde K\mu)_j^i=\theta_j\int_\Omega\Phi^i\e\bar\mu_{t_j}^-+(1-\theta_j)\int_\Omega\Phi^i\e\bar\mu_{t_j}.
 \end{equation*}
Note that only the evaluation of any \cad\:representative $\bar \mu$ of $\mu$ is needed and the expression is well-defined in the sense that it does not depend on the value of $\bar\mu$ at $t=1$ (because $\theta_M=1$). Hence, this definition yields a natural discrete forward operator. More precisely, only the vectors $\bar\mu^+=(\bar\mu^{0,+},\ldots,\bar\mu^{M,+})=(\bar\mu_{t_0},\ldots,\bar\mu_{t_M})$ and  $\bar\mu^-=(\bar\mu^{0,-},\ldots,\bar\mu^{M,-})=(\bar\mu_{t_0}^-,\ldots,\bar\mu_{t_M}^-)$, which may be seen as time samples of a \cad\:curve in $\R_+\mathcal{D}_W(\Omega)$ representing the right and left `traces', are considered. For any pair of vectors $\bar\nu=(\bar\nu^+,\bar\nu^-)=((\bar\nu^{0,+},\ldots,\bar\nu^{M,+}),(\bar\nu^{0,-},\ldots,\bar\nu^{M,-}))\in\R_+(\probs(\Omega)^{M+1}\times\probs(\Omega)^{M+1})$, we set (with a slight abuse of notation)
\begin{equation*}
( K_0\bar\nu)^i_j=\theta_j\int_\Omega\Phi^i\e\bar\nu^{j,-}+(1-\theta_j)\int_\Omega\Phi^i\e\bar\nu^{j,+}.
\end{equation*}
Similarly, given samples $(\bar\nu^+,\bar\nu^-)$ of a \cad\:representative of $\nu=\delta_\gamma$ with $\gamma\in BV(I;\Omega)$, we write $\bar\gamma=(\bar\gamma^+,\bar\gamma^-)$ with $\bar\gamma^+=(\bar\gamma^{0,+},\ldots,\bar\gamma^{M,+})=(\bar\gamma_{t_0},\ldots,\bar\gamma_{t_M}),\bar\gamma^-=(\bar\gamma^{0,-},\ldots,\bar\gamma^{M,-})=(\bar\gamma_{t_0}^{-},\ldots,\bar\gamma_{t_M}^{-})$ and set
\begin{equation*}
( K_0\bar\gamma)^i_j=\theta_j\Phi^i(\bar\gamma^{j,-})+(1-\theta_j)\Phi^i(\bar\gamma^{j,+}).
\end{equation*}
Note that, in principle, each $\Phi^i$ may also be evaluated at values of other reasonable representatives of $\gamma$, e.g.\:representatives whose evaluations at the $t_j$ lie in between $\bar\gamma_{t_j}^-$ and $\bar\gamma_{t_j}^+$ (recall that $\Omega$ is convex).
Since we are particularly interested in \cad\:curves, we use the construction from above.

Next, let us apply the above procedure to discretize the functional $\gamma\mapsto a(\gamma)\int_Ip^k(\gamma)\e\Leb$ in the insertion step of \cref{AlgDynamic}, see \cref{prob:newinsertionstep}. First, we define
\begin{equation*}
p(\mu)=-\widetilde K_*\mathrm{D}\fid(\widetilde K\mu)\qquad\textup{and discretize}\qquad \langle p(\mu),\nu\rangle_{L_w^2}=-\mathrm{D}\fid(\widetilde K\mu):\widetilde K\nu
\end{equation*}
for arbitrary $\mu,\nu\in\V$, where $:$ denotes the Frobenius inner product on $Y=\R^{L\times (M+1)}$. We see directly that the latter expression depends continuously on the measurements $\widetilde K\mu$ and $\widetilde K\nu$ (provided that $\mathrm{D}\fid$ is continuous, which is satisfied by \cref{NumSetup}). Therefore, in the spirit of our discretization approach via temporal deblurring in the data space, we discretize $\langle p(\mu),\nu\rangle_{L_w^2}$ as 
\begin{equation*}
\lim_{\delta\to 0}\langle p(\mu),\nu\rangle_{L_w^2}.
\end{equation*}
In particular, if $\mu=\mu^k$ is an iterate of \cref{AlgDynamic}, then we discretize the functional in the insertion step (\cref{prob:newinsertionstep}) by (recall that $p^k=p(\mu^k)$)
\begin{equation}
\label{eq:discreins}
    D^k_0(\bar\gamma^+, \bar\gamma^-) = a_0(\alpha,\beta,\bar\gamma^+,\bar\gamma^-)\lim_{\delta\to 0}\langle p^k,\delta_\gamma\rangle_{L_w^2},
\end{equation}
where $\gamma\in BV(I;\Omega)$ is arbitrary with $\bar\gamma_{t_j}^{\pm}=\bar\gamma^{j,\pm}$ and
\begin{equation}\label{eq:disca}
a_0(\alpha,\beta,\bar\gamma^+,\bar\gamma^-) = \left(\alpha + \beta\sum_{j=0}^{M-1} (|\bar \gamma^{j,+}-\bar\gamma^{j+1,-}|+|\bar \gamma^{j,+}-\bar\gamma^{j,-}|)\right)^{-1}
\end{equation}
is a discretization of $a(\gamma)$. 
Indeed, the right-hand side in \cref{eq:discreins} does not depend on the choice of $\gamma$ (\cref{VanishingBlur}):
\begin{equation*}
 D^k_0(\bar\gamma^+, \bar\gamma^-) = -a_0(\alpha,\beta,\bar\gamma^+,\bar\gamma^-)\sum_{i=1}^L\sum_{j=0}^M(\mathrm{D}\fid(K_0\mu^k))_{i,j}(\theta_j\Phi^i(\bar\gamma^{j,-})+(1-\theta_j)\Phi^i(\bar\gamma^{j,+})).
\end{equation*}
\subsection{Implementation details}
\label{ImplDet}
This section is devoted to the implementation of \cref{AlgDynamic}. We limit ourselves to the one-dimensional case. It is well-known that in FC-GCG algorithms, the numerical bottleneck lies in the efficient solution of the (in general non-convex) insertion step (\cref{prob2:insertion}), see, for example, \cite[Section 5.1]{KCFR22} and \cite[Section 1.1]{bredies2024asymptotic}. In particular, this optimization is extremely challenging in higher-dimensional domains.
However, we believe that a numerical realization of \cref{AlgDynamic} in the case $n=2$ is possible by transferring the approaches in \cite{KCFR22,duval2024dynamical}. Further, it is likely that \cite{duval2024dynamical} can be applied to speed up the computations in the insertion step.

First, let us specify our (discrete) forward operator. Assume that $n=1$ and fix points $x_1,\ldots,x_L$ in some compact interval $\Omega\subset\R$. Further, pick a discretization $0=t_0<\ldots<t_M=1$ of the time interval $I=[0,1]$. Given $C_i,\sigma_i >0$, we define (truncated) Gaussian kernels $\Phi^i\in C(\Omega)$ (independent of $t$) by
\begin{equation}
\label{eq:GaussianKernels}
\Phi^i(x) = \frac{C_i}{\sigma_i}e^{\frac{-|x-x_i|^2}{2\sigma_i^2}} \qquad \text{for } \qquad i =1,\ldots,L,
\end{equation}
where $\sigma_i^2$ indicates variance and $C_i$ can be chosen arbitrarily (e.g.\:such that $\int_\Omega\Phi^i\e\Leb=1$). We use these functions in the definition of $\widetilde K=\widetilde K(\delta)$ (see \cref{eq:linex2}) and define $K_0$ as in \cref{DerivationForward}. Recall that 
\begin{equation*}
( K_0\bar\gamma)^i_j=\theta_j\Phi^i(\bar\gamma^{j,-})+(1-\theta_j)\Phi^i(\bar\gamma^{j,+})
\end{equation*}
whenever $\bar\gamma=(\bar\gamma^+,\bar\gamma^-)$ with $\bar\gamma^+=(\bar\gamma^{0,+},\ldots,\bar\gamma^{M,+})$ and $\bar\gamma^-=(\bar\gamma^{0,-},\ldots,\bar\gamma^{M,-})$ corresponds to time samples of a \cad\:representative of $\gamma\in BV(I;\Omega)$ respectively $\delta_\gamma\in BV(I;\wass)$. As mentioned in \cref{DerivationForward}, these vectors can be interpreted as the left and right `traces' of this representative at the discrete time points $t_j$. In particular, we enforce any reconstructed jump to take place at these points. In the remainder, we take $\theta_j=0$ for all $j\neq M$ (recall that $\theta_M=1$), which yields evaluation of $\Phi^i$ at the right-hand limits for $t=t_0,\ldots,t_{M-1}$ and evaluation at the left-hand limit for $t=t_M=1$ (which is reasonable by the non-uniqueness of a \cad\:representative of $\gamma$ at $t=1$). 
\begin{rem}[Space and time discretization]
In order to reduce the computational cost, we consider the time discretization $t_0,\ldots,t_M$. This is standard for such particular algorithms \cite{KCFR22, duval2024dynamical}. We stress that the discretization $x_1,\ldots,x_L$ corresponds to the definition of $\widetilde K$ (respectively $K_0$). In particular, a spatial numerical discretization is not needed because \cref{AlgDynamic} is grid-free (thus it allows for super-resolution).
\end{rem}
We consider the fidelity $\mathcal{F}:\mathbb{R}^{L \times (M+1)} \to \R_+$ defined by
\begin{equation*}
\mathcal{F}_f(y) = \frac{1}{2(M+1)}\|y - f\|^2_{F} ,
\end{equation*}
where $f \in \mathbb{R}^{L \times (M+1)}$ represents given reference data and $\|\cdot\|_F$ denotes the Frobenius norm. 
\begin{examp}[Formula for $D^k_0(\bar\gamma^+, \bar\gamma^-)$]
In the above setting, we have (\cref{VanishingBlur})
\begin{multline*}
D^k_0(\bar\gamma^+, \bar\gamma^-)=-\frac{a_0(\alpha,\beta,\bar\gamma^+, \bar\gamma^-)}{M+1}\sum_{j=0}^{M-1}\sum_{i_1=1}^{L}\Phi^{i_1}(\bar\gamma^{j,+})\left(\sum_{\ell_1=1}^{N}\lambda_{\ell_1}^k\Phi^{i_1}((\bar\gamma_\ell^k)^{j,+})-f_{j}^i\right)\\
-\frac{a_0(\alpha,\beta,\bar\gamma^+, \bar\gamma^-)}{M+1}\sum_{i_2=1}^{L}\Phi^{i_2}(\bar\gamma^{M,-})\left(\sum_{\ell_2=1}^{N}\lambda_{\ell_2}^k\Phi^{i_2}((\bar\gamma_\ell^k)^{M,-})-f_M^i\right),
\end{multline*}
where $\mu^k=\sum_{\ell=1}^N\lambda_\ell^k\delta_{\gamma_\ell^k}$ is an iterate of \cref{AlgDynamic}.
\end{examp}
In order to maximize $(\bar\gamma^+, \bar\gamma^-)\mapsto D^k_0(\bar\gamma^+, \bar\gamma^-)$, we adopt the multi-start gradient descent approach in \cite[Section 5.1]{KCFR22}. First, we sample from the uniform distribution in $\Omega$ two ($M+1$)-tuples $\bar\gamma_0^+ = (\bar \gamma_0^{0,+}, \ldots, \bar \gamma_0^{M,+})$ and $\bar\gamma_0^-  = (\bar \gamma_0^{0,-}, \ldots, \bar \gamma_0^{M,-})$. 
Then, using $\bar\gamma_0^+$ and $\bar\gamma_0^-$ as initializations, we run a gradient ascent algorithm\footnote{The Euclidean norm in \cref{eq:disca} is approximated by $\eta_\varepsilon(z) = \sqrt{|z|^2 + \varepsilon}$.} to maximize functional $D_0^k$ obtaining as output vectors $\bar\gamma_{0,*}^+$, $\bar\gamma_{0,*}^-$. 
We repeat this procedure for a fixed amount of random initializations $q=0,...,Q$. 
Then we store the vectors $\bar\gamma_{q_{max},*}^+$, $\bar\gamma_{q_{max},*}^-$ which yield the maximum value of $q\mapsto D^k_0(\bar\gamma_{q,*}^+,\bar\gamma_{q,*}^-)$. These vectors will be the output of the insertion step. In order to balance between accuracy and efficiency, it is crucial to carefully choose the number of initializations $Q$ of the gradient ascent algorithm. In all the experiments in \cref{NumRes}, we take $Q=150$.  
\subsection{Numerical experiments}
\label{NumRes}
In the following experiments, we use $\Omega = [0,5]$ and equidistant time discretization $0=t_0 < \ldots < t_{M} =1$ with $M = 30$. 
We employ the forward operator $\widetilde K$ given by \cref{eq:linex2,eq:GaussianKernels} with equidistant points $0=x_1<\ldots<x_{100}=5$ in $\Omega$ and $C_i\equiv \frac{1}{\sqrt{2\pi}},\sigma_i^2 \equiv 0.02$.

\underline{Three curves, no noise:}
In the first numerical experiment, we aim to reconstruct the ground truth $\bar\mu^\dagger\in\mathcal{D}_E(\Omega)$ given by 
\begin{equation}
\label{eq:gt1}
    \bar\mu^\dagger = \delta_{\bar\gamma_1^\dagger} + \delta_{\bar\gamma_2^\dagger} + \delta_{\bar\gamma_3^\dagger},
\end{equation}
where 
\begin{equation}
\label{eq:gt2}
    \bar\gamma_1^\dagger(t) =  t + 3.5, \qquad \bar\gamma_2^\dagger(t) = \sqrt{t} + 2.5 , \qquad\textup{and}\qquad \bar\gamma_3^\dagger(t) =
    \begin{cases*}
        1+t^2&if $t<0.5$,\\
        2+t^2&if $t\geq 0.5$.
    \end{cases*}
\end{equation}
A discretization of the ground truth $\bar\mu^\dagger$ is given in \cref{eq:gt}. 
\begin{figure}[h!]
		\centering
		\includegraphics[width=8.2cm,height=6.4cm]{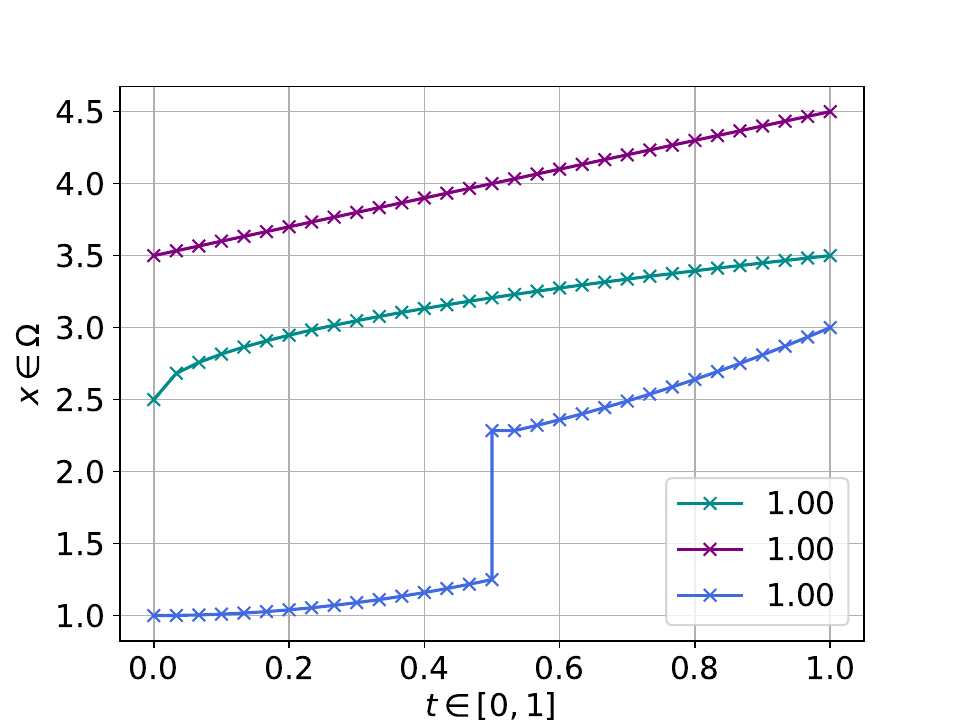}
        \caption{Discretization of the ground truth $\bar\mu^\dagger$ defined by \cref{eq:gt1,eq:gt2}. The legend shows the weights associated with the \cad\:curves $\delta_{\bar\gamma_i^\dagger}$.}
        \label{eq:gt}
\end{figure}

We consider the reference data $f = K_0 \bar\mu^\dagger \in \mathbb{R}^{L \times (M+1)}$ (recall that $K_0$ equals the pointwise limit of $\widetilde K=\widetilde K(\delta)$) and choose regularization parameters $\alpha = 5$ and $\beta = 2$.
Our reconstruction is depicted in \cref{sf:Exp1a}. \Cref{sf:Exp1b} shows the convergence of the residuals (\cref{eq:residuals}).
\begin{figure}[t]
		\centering
		\begin{subfigure}[b]{0.45\textwidth}
			\centering
            \includegraphics[width=8.2cm,height=6.4cm]{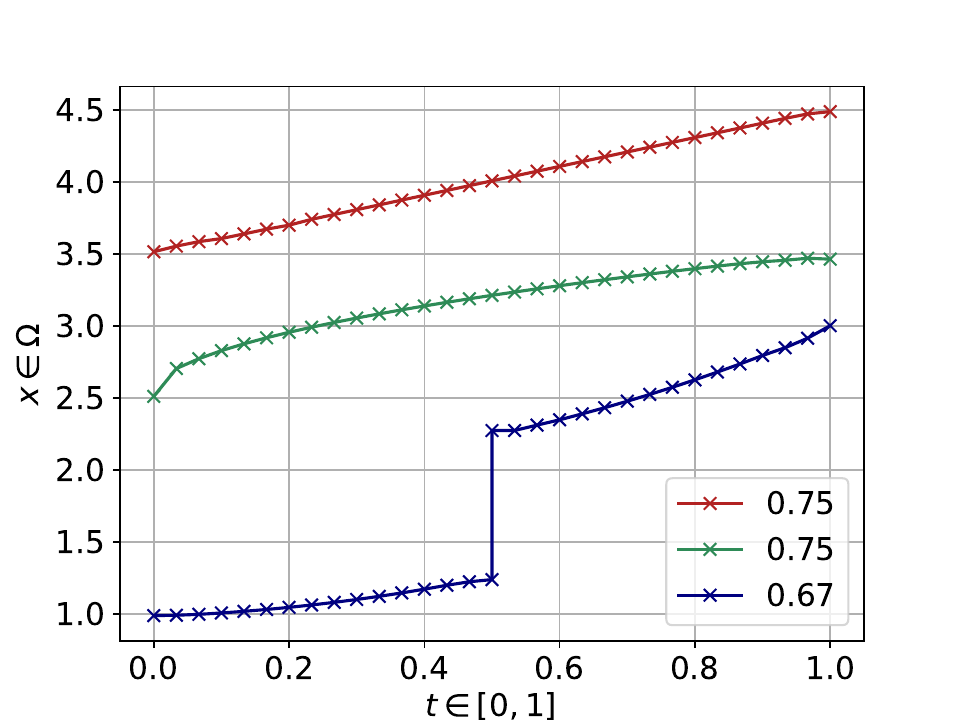} 
			\caption{$\alpha = 5$ and $\beta = 2$.}
			\label{sf:Exp1a}
		\end{subfigure}
        \begin{subfigure}[b]{0.45\textwidth}	
            \centering
\includegraphics[width=8.2cm,height=6.4cm,trim=0 0.5cm 0 -1.3cm]{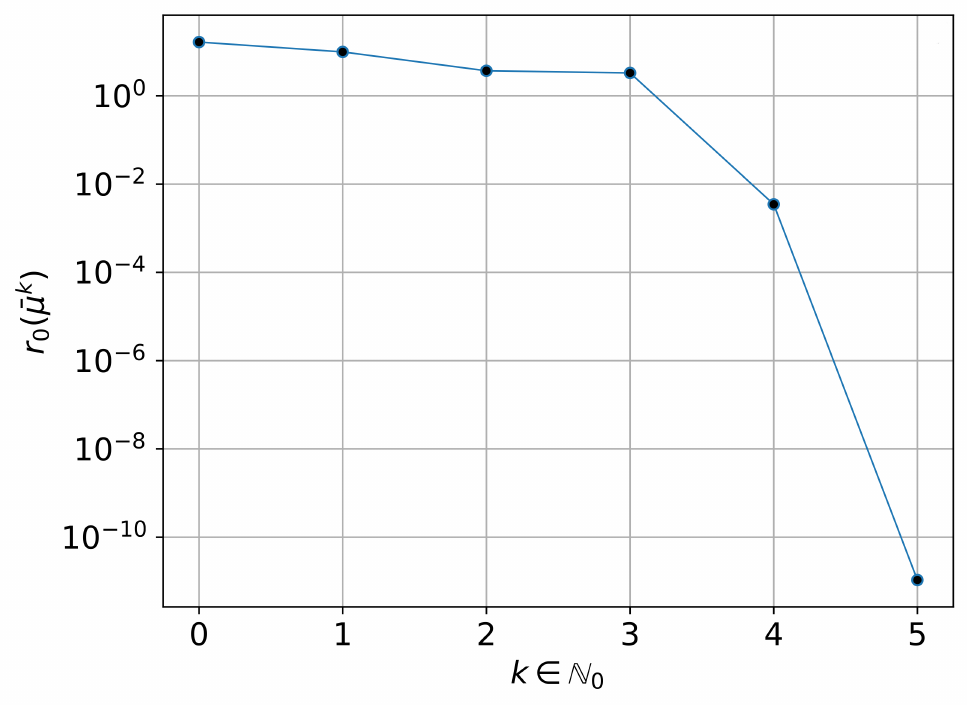}
			\caption{Convergence of residuals ($\log$ scale).}
			\label{sf:Exp1b}
		\end{subfigure}
\caption{Reconstruction of the discretized ground truth from \cref{eq:gt}. Note that it is faithful to the ground truth in the sense that it recovers the jump of $\bar\gamma^\dagger_3$. The effect of regularization is observable in the attenuated weights and decreased variation of each curve (in particular, close to $t = 0$ and $t=1$).}
	\end{figure} 
Note that, since $\min  \mathcal J_{\alpha,\beta}$ is unknown, we approximate the residuals by evaluating our discretization\footnote{It is naturally defined in the spirit of \cref{DerivationForward} using $K_0$ and an approximation of the essential variation as in $a_0$, see \cref{eq:disca}.} of $\mathcal{J}_{\alpha,\beta}$ at the last iterate returned by \cref{AlgDynamic}. 
The same idea was used in \cite{KCFR22}. 
The corresponding residuals are denoted by $r_0(\bar \mu^k)$. 
To highlight the effect of our regularizer $\reg_{\alpha,\beta}$, we perform another reconstruction using $\alpha = 12$ and $\beta = 5$.
It is shown in \cref{sf:Exp2a} (next to the residuals $k\mapsto r_0(\bar\mu^k)$ in \cref{sf:Exp2b}).
\begin{figure}[h!]
\begin{subfigure}[b]{0.45\textwidth}
\centering
\includegraphics[width=8.2cm,height=6.4cm]{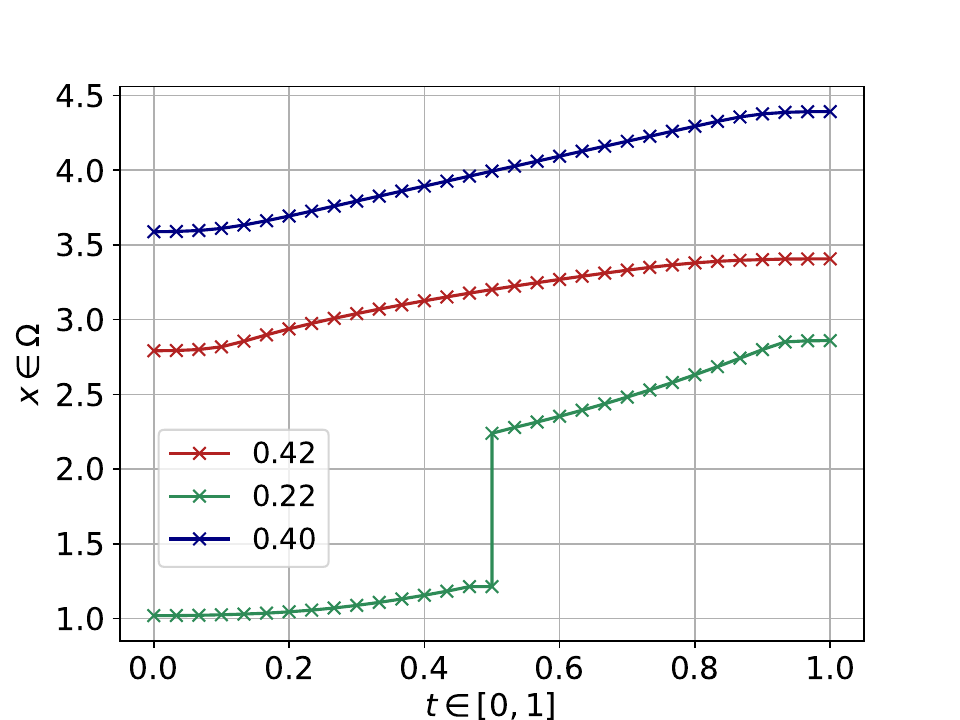}
\caption{$\alpha = 12$ and $\beta = 5$.}
\label{sf:Exp2a}
\end{subfigure}
\begin{subfigure}[b]{0.45\textwidth}
\centering
\includegraphics[trim=0 0.35cm 0 -1.2cm, width=8.2cm,height=6.4cm]{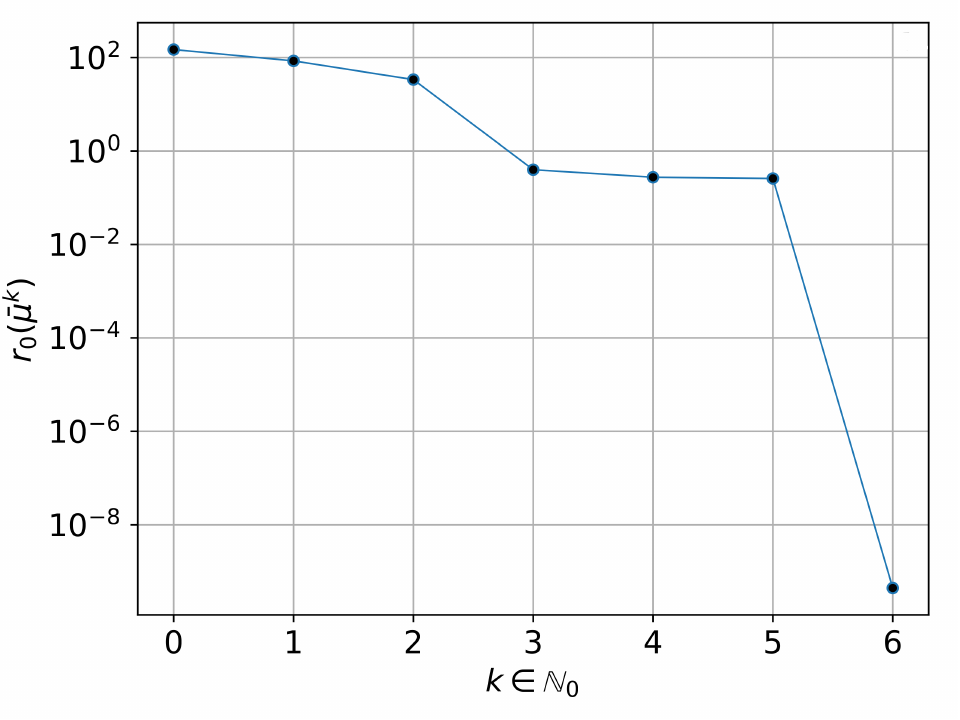}
\caption{Convergence of residuals (log scale).}
\label{sf:Exp2b}
\end{subfigure}
\caption{The larger value of $\beta$ reduces the variation of each of the curves --- close to $t = 0$ and $t=1$ the Diracs move with small velocity. The weights of the reconstructed curves are further scaled down because of the larger value of $\alpha$.}\label{fig:exp2}
\end{figure}

\underline{Three curves with noise:} In the second experiment, we consider the same ground truth as in the previous example (defined by \cref{eq:gt1,eq:gt2}).
Moreover, we perturb the measurements with Gaussian noise.
In particular, we add a matrix $\mathcal{N} \in \mathbb{R}^{L\times (M+1)}$ whose entries are realizations of normally distributed (with zero mean and standard deviation equal to $0.2$) random variables. 
Hence, the measurement $K_0\bar\mu^\dagger + \mathcal{N}$ is used.
Further, we take $\alpha = 5$ and $\beta = 3$. 
In general, we expect that our regularization allows for a reconstruction which is stable with respect to (sufficiently small) additive noise (since our regularizer is a natural extension of `static' regularizers with similar properties).
This is also due to \cref{stability} and the choice of our fidelity term.
The reconstruction is illustrated in \cref{fig:noisy}. 
\begin{figure}[h!]
\begin{subfigure}[b]{0.45\textwidth}
\centering
\includegraphics[width=8.2cm,height=6.4cm]{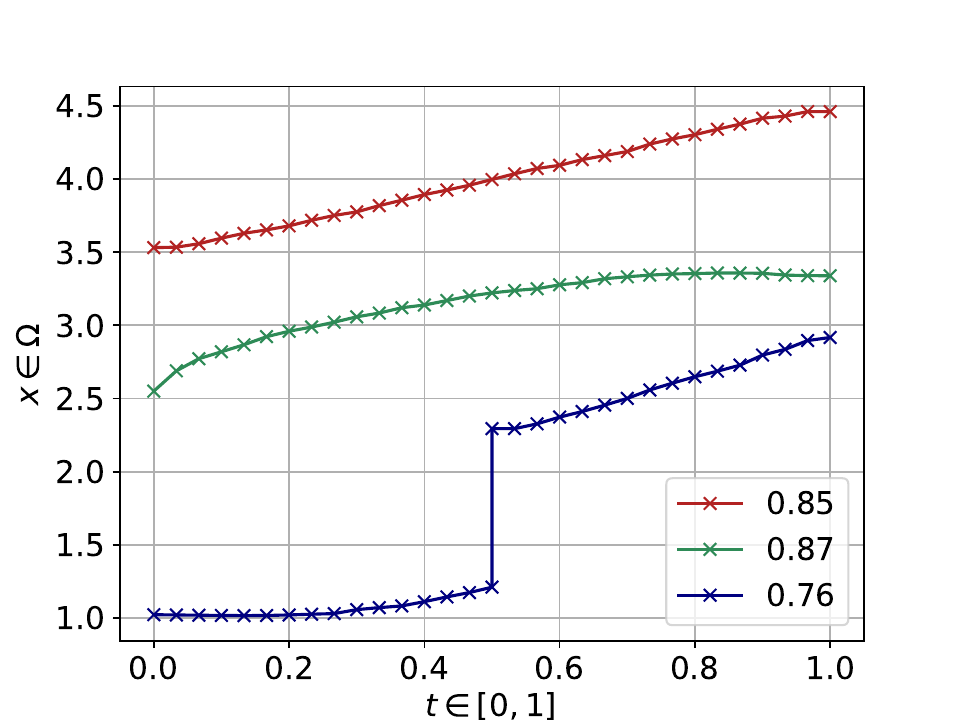}
\caption{$\alpha = 5$ and $\beta = 3$.}
\end{subfigure}
\begin{subfigure}[b]{0.45\textwidth}
\centering
\includegraphics[width=8.2cm,height=6.4cm]{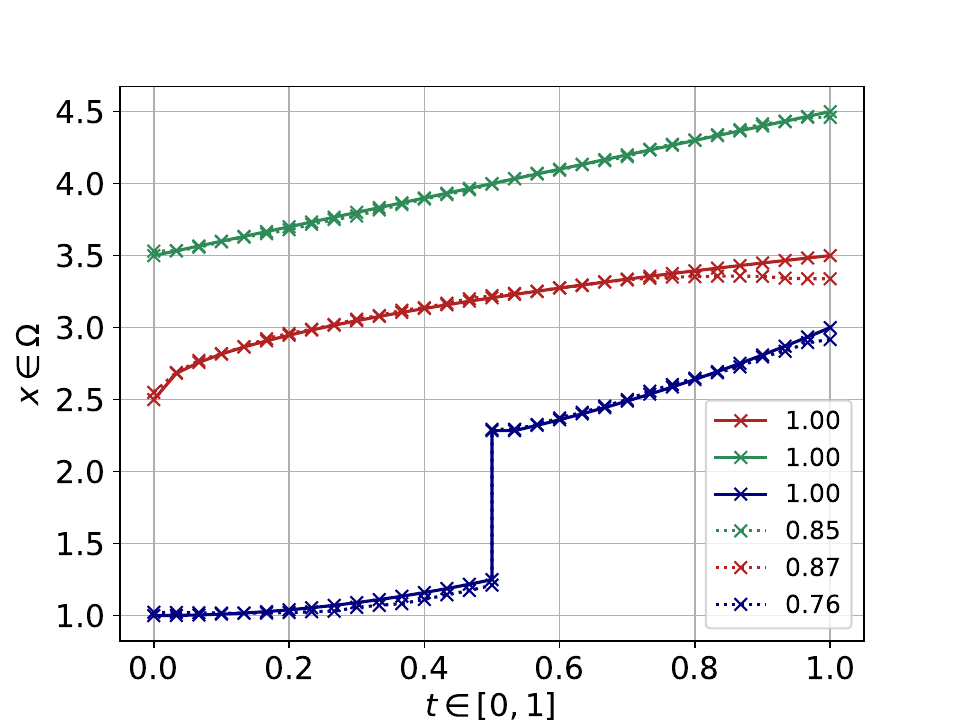}
\caption{Superimposition ground truth \& reconstruction.}
\end{subfigure}
\caption{Reconstruction of $\bar\mu^\dagger$ from $K_0\bar\mu^\dagger + \mathcal{N}$. The superimposition of ground truth (depicted with a solid line) and reconstruction highlights the regularization bias. As expected, the reconstruction is not strongly affected by the noise. Similar regularization effects as in the previous experiments can be observed.  }
\label{fig:noisy}
\end{figure}

\underline{Crossing curves:} In this experiment, we consider the problem of reconstructing two curves which cross each other at time $t = 0.5$. 
More specifically, we use the ground truth given by (see \cref{fig:gtcrossing})
\begin{equation*}
    \bar \mu^\dagger = \delta_{\bar \gamma^\dagger_1} + \delta_{\bar \gamma^\dagger_2}
\end{equation*}
where 
\begin{equation*}
    \bar\gamma^\dagger_1(t) = 1 + 3t\qquad\textup{and} \qquad \bar\gamma^\dagger_2(t) = 4 - 3t.
\end{equation*}
Again, we compute the reference data as $f = K_0 \bar \mu^\dagger \in \mathbb{R}^{L \times (M+1)}$.
The reconstruction is shown in \cref{fig:reconstructioncrossing}.
A similar effect to the one highlighted in \cite[Section 6.2.3]{KCFR22} (where a Wasserstein-$2$-type regularization is considered) can be observed.
\begin{figure}[h!]
\begin{subfigure}[b]{0.45\textwidth}
\centering
\includegraphics[width=8.2cm,height=6.4cm]{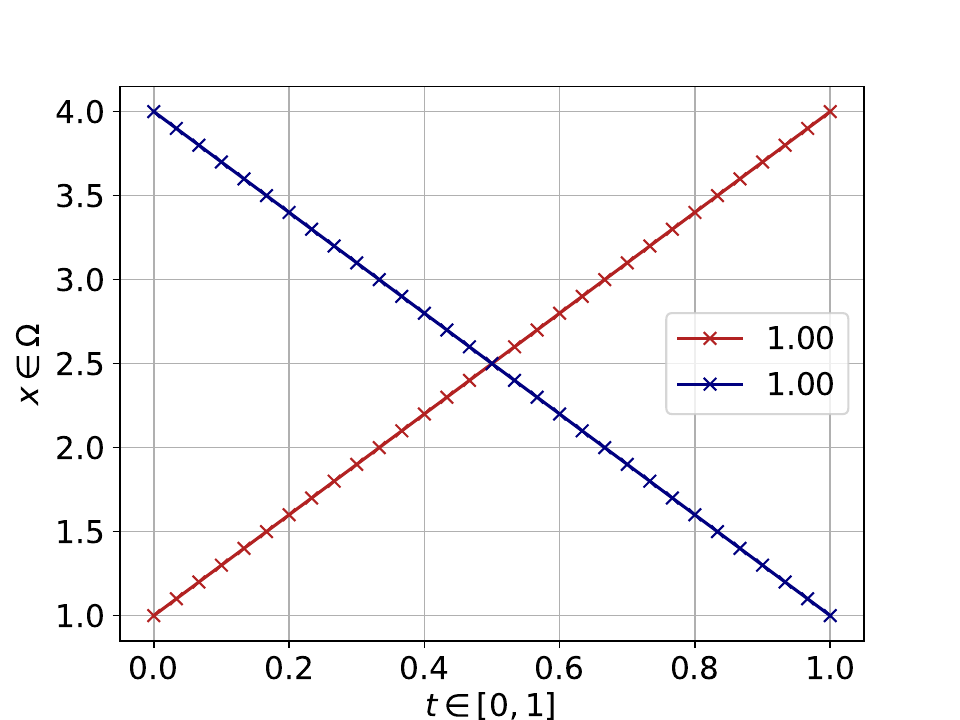}
\caption{Ground truth.}
\label{fig:gtcrossing}
\end{subfigure}
\begin{subfigure}[b]{0.45\textwidth}
\centering
\includegraphics[width=8.2cm,height=6.4cm]{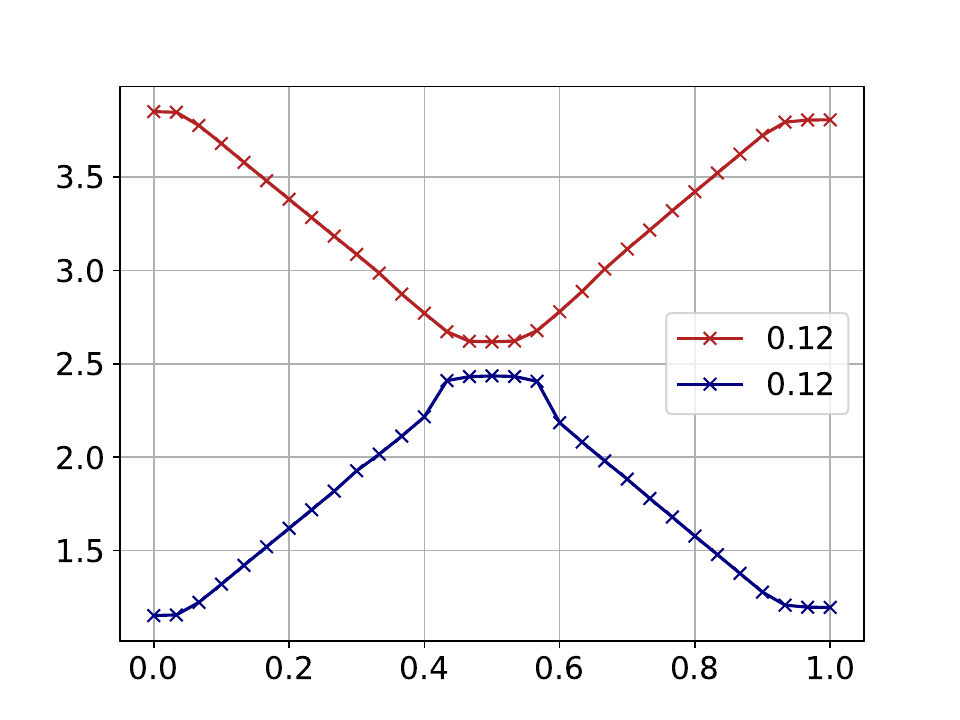}
\caption{Reconstruction for $\alpha = 13$ and $\beta = 5$.}
\label{fig:reconstructioncrossing}
\end{subfigure}
\caption{Reconstruction of two crossing curves. Our algorithm is unable to identify the correct trajectories due to the regularization (which encourages small variation). 
}
\end{figure}

\underline{Non-sparse ground truths:}
Finally, we reconstruct diffuse ground truths which cannot be written as linear combinations of Diracs in elements of $\mathcal{D}_E(\Omega)$. More precisely, we take $\bar \mu^\dagger, \bar \nu^\dagger \in \mathcal{D}_W(\Omega)$ defined by
\begin{equation*}
    \bar \mu^\dagger_t = \mathcal{L} \mres [1 + t, 4 - t] \qquad \textup{and} \qquad \bar \nu^\dagger_t = \mathcal{L} \mres [\zeta_1(t), \zeta_2(t)]
\end{equation*}
with 
\begin{equation*}
    \zeta_1(t) = 
    \begin{cases*}
        1+t&if $t<0.5$,\\
        2+t&if $t\geq 0.5$
    \end{cases*}
\qquad\textup{and}\qquad
\zeta_2(t) = 
\begin{cases*}
        2+t&if $t<0.5$,\\
        3+t&if $t\geq 0.5$.
    \end{cases*}    
\end{equation*}
 We define $f = K_0 \bar \mu^\dagger ,g = K_0 \bar \nu^\dagger \in \mathbb{R}^{L\times (M+1)}$ as the corresponding reference data.
 Note that $\bar\mu^\dagger$ is absolutely continuous while $\bar\nu^\dagger$ jumps at $t = 0.5$.
 We use regularization parameters $\alpha = 3$ and $\beta = 2$ for the reconstruction of $\bar \mu^\dagger$ and $\alpha = 5$ and $\beta = 2$ for the reconstruction of $\bar \nu^\dagger$.
 The results are shown in \cref{fig:exp2}. Recall that \cref{AlgDynamic} yields sparse reconstructions despite our choice of the ground truths.
\begin{figure}[h]
\begin{subfigure}[b]{0.45\textwidth}
\centering
\includegraphics[width=8.2cm,height=6.4cm]{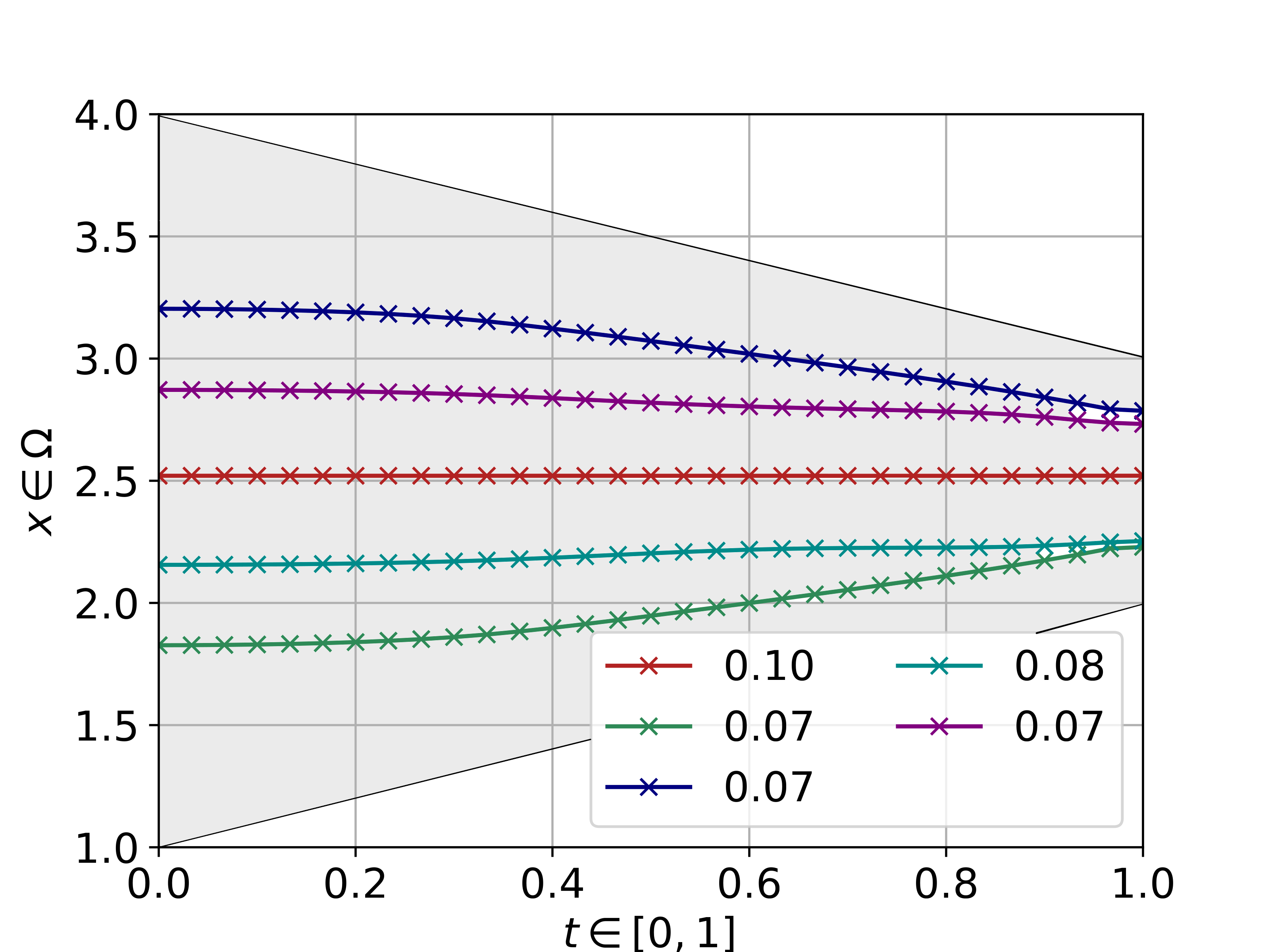}
\caption{Reconstruction of $\bar \mu^\dagger$ ($\alpha = 3, \beta = 2$).}
\end{subfigure}
\begin{subfigure}[b]{0.45\textwidth}
\centering
\includegraphics[width=8.2cm,height=6.4cm]
{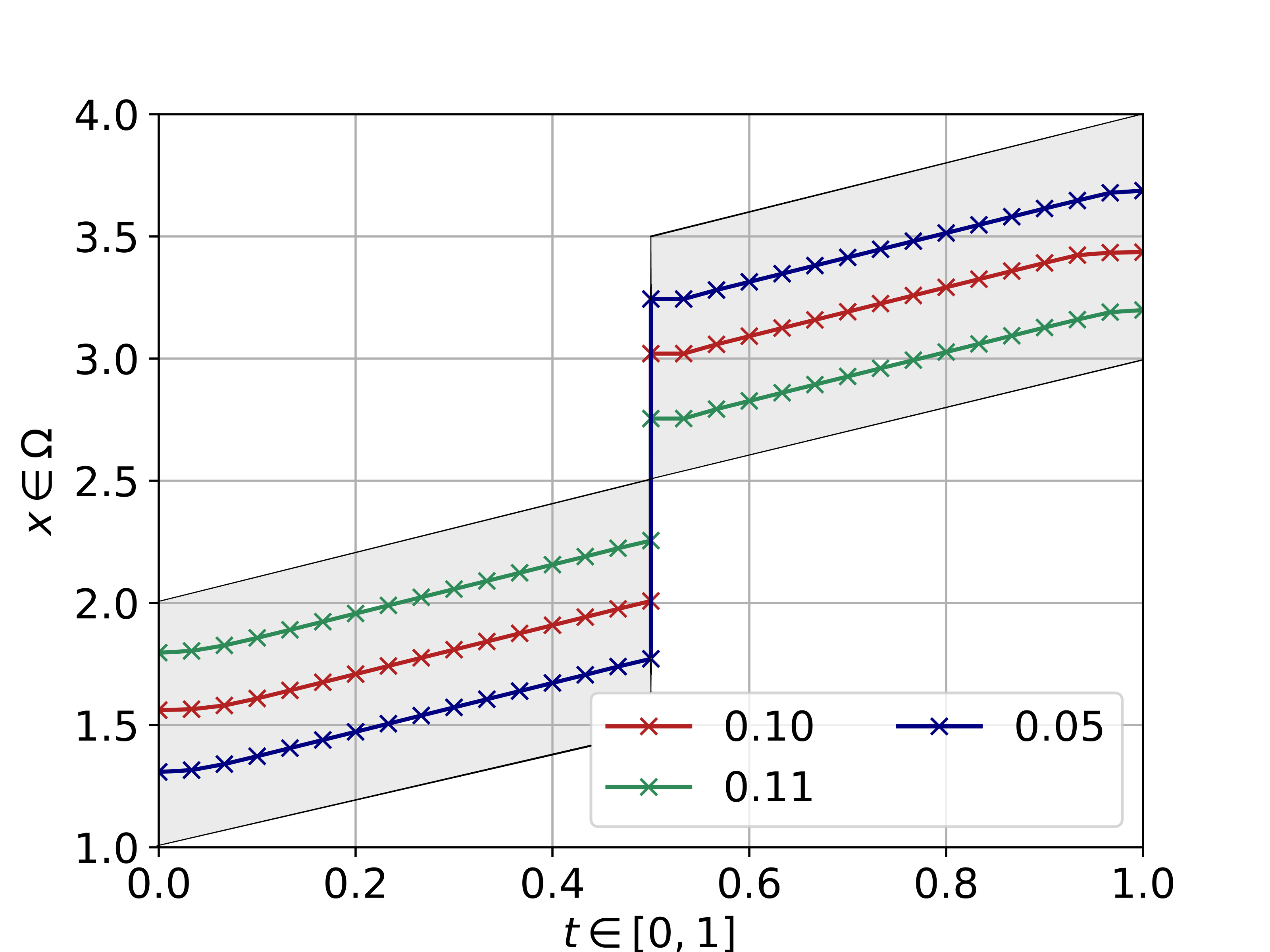}
\caption{Reconstruction of $\bar \nu^\dagger$ ($\alpha = 5$, $\beta = 2$).}
\end{subfigure}
\caption{Reconstruction of diffuse ground truths. As expected, curve $\bar \mu^\dagger$ is approximated by only using absolutely continuous curves while $\bar \nu^\dagger$ is approximated by curves with jumps.}
\label{fig:exp2}
\end{figure}

    \section{Acknowledgements}
    M.C.'s research was supported by the NWO-M1 grant Curve Ensemble Gradient Descents for Sparse Dynamic Problems (Grant Number OCENW.M.22.302).
	J.L.’s research was supported by the JSPS Postdoctoral Fellowship for Research in Japan and the KAKENHI Grant-in-Aid for Scientific Research (Grant Number JP24KF0215) awarded by the Japan Society for the Promotion of Science. J.L.\:would like to thank Carola-Bibiane Schönlieb for the hosting of a NoMADS secondment at the University of Cambridge, funded by the European Union's Horizon 2020 research and innovation programme (Grant Number 777826), during which this research collaboration began. J.L.’s research was also supported by the Deutsche Forschungsgemeinschaft (DFG, German Research Foundation) under the priority program SPP 1962, grant WI 4654/1-1, and under Germany’s Excellence Strategy EXC 2044-390685587, Mathematics Münster: Dynamics-Geometry-Structure. 
	\printbibliography
\end{document}